\documentclass[reqno, 11pt, letterpaper]{amsart}

\usepackage{amsmath}
\usepackage{amsfonts}
\usepackage{amssymb}
\usepackage{graphicx}
\usepackage{amsthm,color,yfonts,cite}
\usepackage{paralist}
\usepackage{hyperref}
\usepackage{physics}
\usepackage{bbm}
 \usepackage{bm}
\usepackage{comment}

\newtheorem{theorem}{Theorem}

\newtheorem{proposition}[theorem]{Proposition}
\newtheorem{lemma}[theorem]{Lemma}

\theoremstyle{remark}
\newtheorem{remark}[theorem]{Remark}

\usepackage{wrapfig}
\usepackage{tikz}
\usetikzlibrary{arrows,calc,decorations.pathreplacing}
\definecolor{light-gray1}{gray}{0.90}
\definecolor{light-gray2}{gray}{0.80}
\definecolor{light-gray3}{gray}{0.60}

\numberwithin{equation}{section}

\numberwithin{theorem}{section}

\numberwithin{table}{section}

\numberwithin{figure}{section}

\ifx\pdfoutput\undefined
  \DeclareGraphicsExtensions{.pstex, .eps}
\else
  \ifx\pdfoutput\relax
    \DeclareGraphicsExtensions{.pstex, .eps}
  \else
    \ifnum\pdfoutput>0
      \DeclareGraphicsExtensions{.pdf}
    \else
      \DeclareGraphicsExtensions{.pstex, .eps}
    \fi
  \fi
\fi

\title[Modified scattering for the Vlasov-Riesz system]{Modified scattering for the Vlasov-Riesz system with long-range interactions}

\date{\today}
\author[Y. Hong]{Younghun Hong}
\address{Department of Mathematics, Chung-Ang University, Seoul 06974, Korea}
\email{yhhong@cau.ac.kr}

\author[S. Pankavich]{Stephen Pankavich}
\address{Department of Applied Mathematics and Statistics, Colorado School of Mines, Golden, CO 80401, USA}
\email{pankavic@mines.edu}
\begin{document}

\begin{abstract}
We study the long-time asymptotic behavior of small-data solutions to the three-dimensional Vlasov--Riesz system with the inverse power-law potential $\lambda |x|^{-\alpha}$ in the strictly long-range regime ($0 < \alpha < 1$). By introducing finite- and infinite-time modified wave operators for the characteristic flows, we describe the asymptotic dynamics via convergence to an effective profile along a suitably modified reference flow, and establish modified scattering of solutions. Our proof relies mainly on ODE techniques for the characteristic flows, while also using PDE methods for weighted $W^{1,\infty}$-bounds. Compared with the earlier result \cite{HuangKwon}, our Lagrangian approach extends modified scattering to the broader regime $\frac{1}{2}<\alpha<1$ and provides a distinct and more robust argument.
\end{abstract}

\maketitle


\section{Introduction}

\subsection{Model description}

We consider the Vlasov equation
\begin{equation}\label{eq: VR}
\left\{
\begin{aligned}
\partial_t f + v \cdot \nabla_x f + \mathbf{E}_f \cdot \nabla_v f &= 0, \\
f(0) &= f_0,
\end{aligned}
\right.
\end{equation}
where
\[
f = f(t,x,v) : [0,\infty) \times \mathbb{R}^3 \times \mathbb{R}^3 \to [0,\infty)
\]
denotes the particle distribution function in phase space. The self-consistent force field
\[
\mathbf{E}_f = \mathbf{E}_f(t,x) : [0,\infty) \times \mathbb{R}^3 \to \mathbb{R}^3
\]
is defined by
\begin{equation}\label{eq: VR force field}
\mathbf{E}_f(t,x)
= - \nabla (w * \rho_f)(t,x)
= - \int_{\mathbb{R}^3} \nabla w(x-y)\, \rho_f(t,y)\, dy,
\end{equation}
where \(w : \mathbb{R}^3 \to \mathbb{R}\) is the interaction potential. 
For a distribution \(h = h(x,v) : \mathbb{R}^3 \times \mathbb{R}^3 \to [0,\infty)\), the associated spatial density is
\[
\rho_h(x) := \int_{\mathbb{R}^3} h(x,v)\, dv.
\]

The equation \eqref{eq: VR} describes the mean-field dynamics of a particle system, such as a plasma or a galactic system, interacting via the pairwise potential \( w \). In this article, we assume
\begin{equation}\label{eq: w inverse power law assumption}
w(x) = \frac{\lambda}{|x|^\alpha},
\end{equation}
with \( 0 < \alpha < 3 \). The coupling constant \( \lambda \neq 0 \) is arbitrary in principlem but upon normalizing, we restrict to the cases \( \lambda = 1 \) (repulsive) and \( \lambda = -1 \) (attractive). Under this assumption, \eqref{eq: VR} is called the \emph{Vlasov--Riesz system}, as the convolution kernel \( w(x-x') \) coincides, up to a multiplicative constant, with the Riesz potential \( |\nabla|^{-(3-\alpha)} \). More precisely, the force field can be written as \(\mathbf{E}_f = -\nabla(c_\alpha \lambda\, |\nabla|^{-(3-\alpha)} \rho_f)\). 

When \( \alpha = 1 \), the system reduces to the classical \emph{Vlasov--Poisson system}, which has been studied extensively throughout the literature. In particular, global existence of small data solutions was first established by Bardos and Degond in \cite{BD}. Thereafter, large data, global-in-time solutions were discovered independently by Pfaffelmoser \cite{Pfaffelmoser} and Lions and Perthame \cite{Lions-Perthame}, and the methods of the former author were subsequently significantly refined by Schaeffer \cite{SchaefferVP}.

Although much is known concerning the well-posedness and behavior of solutions to Vlasov-Poisson, significantly less has been determined about the more general Vlasov-Riesz system. Recently,  Choi and Jeong \cite{CJ} proved the local-in-time well-posedness of this system in a weighted Sobolev space for $\alpha < \frac{9}{4}$, while in the attractive case, they obtained the formation of a finite-time singularity for solutions when $2 \leq \alpha < \frac{9}{4}$. For general results concerning Riesz interactions, see \cite{Lewin,NRS} and the references therein.

The long-time asymptotic behavior of solutions to \eqref{eq: VR} is also of particular interest. 
When the interaction potential \( w \) takes the form \eqref{eq: w inverse power law assumption} with \( \alpha > 1 \), it is referred to as \emph{short-range}, and small-data solutions are expected to approach free transport as \( t \to \infty \). 
This phenomenon, known as \emph{scattering}, was rigorously established for \( 1 < \alpha < 2 \) in \cite{HuangKwon, HongPankavich}. In contrast, for \emph{long-range} interactions, namely \( 0 < \alpha \le 1 \), classical scattering fails, as shown in \cite{ChoiHa}. 
Nevertheless, the asymptotic dynamics can still be described through a suitable modification of scattering, known as \emph{modified scattering}. 
For the Vlasov--Poisson system, i.e. in the borderline case \( \alpha = 1 \), modified scattering with a logarithmic correction was first proved by Choi and Kwon~\cite{ChoiKwon}. 
An alternative, and significantly more constructive, proof was later established by Ionescu, Pausader, Wang, and Widmayer~\cite{IPWW22}. See also \cite{Bigorgne, Breton, Glassey, Pankavich1, Pankavich2, Pankavich3, PBA, Wang23} for related developments in kinetic equations, including scattering in higher dimensions and modified scattering of small data solutions to the relativistic Vlasov-Maxwell system. 
Additionally, modified scattering has been better understood in the context of dispersive PDEs for many years. We refer the reader to the classical results of Ginibre and Velo \cite{GV1, GV2}, Hayashi and Naumkin \cite{HayashiNaumkin2001}, and Nakanishi \cite{N1, N2} for more information.
That being said, much less is known for \eqref{eq: VR} in the strictly long-range regime (\( 0 < \alpha < 1 \)). 
To the best of the authors' knowledge, the only result in this regime is due to Huang and Kwon~\cite{HuangKwon}, who recently proved modified scattering for \( \alpha \) sufficiently close to \(1\).
Hence, one of the main novelties of the current paper is the extension of these results to a wider class of potentials, satisfying $\frac{1}{2} < \alpha < 1$.

\subsection{Statement of the main result}

Throughout, we study the long-time asymptotic behavior of small-data solutions to the Vlasov--Riesz system \eqref{eq: VR} in the strictly long-range regime \( \frac{1}{2} < \alpha < 1 \). 
Prior stating the main result, we recall several preliminaries concerning the small-data theory for \eqref{eq: VR}. 
Let \( 0 < \alpha < 1 \) and assume that the initial data \( f_0 \ge 0 \) satisfies
\begin{equation}\label{eq: small data condition for decay only}
\| f_0 \|_{W_{x,v}^{1,1}(\mathbb{R}^6)\cap W_{x,v}^{1,\infty}(\mathbb{R}^6)}
+ \| f_0 \|_{W_x^{1,1}(\mathbb{R}^3; W_v^{1,\infty}(\mathbb{R}^3))\cap W_v^{1,1}(\mathbb{R}^3; W_x^{1,\infty}(\mathbb{R}^3))}
\le \eta_*,
\end{equation}
for a sufficiently small \( \eta_* > 0 \); see \eqref{eq: mixed norm definition} for the precise definition of the mixed norms. 
Under this condition, the Vlasov--Riesz system \eqref{eq: VR} admits a unique global-in-time solution \(f(t) \in L^1_{x,v}(\mathbb{R}^6) \cap L^\infty_{x,v}(\mathbb{R}^6)\) with initial data \( f_0 \), satisfying the dispersion bounds
\begin{equation}\label{eq: dispersion estimates-0}
\sup_{t \ge 0}
\Big(
\langle t \rangle^{1 + \alpha}
\| \mathbf{E}_f(t) \|_{L_x^\infty(\mathbb{R}^3)}
+ \langle t \rangle^{2 + \alpha}
\| \nabla \mathbf{E}_f(t) \|_{L_x^\infty(\mathbb{R}^3)}
+ \langle t \rangle^{3}
\| \nabla^2 \mathbf{E}_f(t) \|_{L_x^\infty(\mathbb{R}^3)}
\Big)
\lesssim \eta_*,
\end{equation}
as we show within Theorem~\ref{thm: global well-posedness}. Moreover, the PDE \eqref{eq: VR} is equivalent to the forward-in-time Hamiltonian flow \( \Phi(t) \) on \( \mathbb{R}^3 \times \mathbb{R}^3 \), defined by
\begin{equation}\label{eq: Hamiltonian flow Phi(t)}
\Phi(t)(x,v)
:= \big( \mathcal{X}(t,0,x,v),\, \mathcal{V}(t,0,x,v) \big),
\end{equation}
where the characteristic curves \( (\mathcal{X}, \mathcal{V}) \) solve
\begin{equation}\label{eq: forward characteristic ODE for VP-0}
\left\{
\begin{aligned}
\partial_t \mathcal{X}(t,0,x,v) &= \mathcal{V}(t,0,x,v), \\
\partial_t \mathcal{V}(t,0,x,v) &= \mathbf{E}_f \big( t, \mathcal{X}(t,0,x,v) \big), \\
\mathcal{X}(0,0,x,v) &= x, \\
\mathcal{V}(0,0,x,v) &= v.
\end{aligned}
\right.
\end{equation}
Consequently, for a given initial distribution \( f_0 \), the pushforward measure \( \Phi(t)_\# f_0 \) yields a global weak solution to \eqref{eq: VR}; see Remark~\ref{remark: PDE ODE relation} for further details.

By the dispersion bound \eqref{eq: dispersion estimates-0}, the nonlinear Hamiltonian flow \( \Phi(t) \) of the Vlasov--Riesz system can, to leading order, be approximated by the reference flow \( \Phi^{\mathrm{ref}}(t) \) defined by
\[
\Phi^{\mathrm{ref}}(t)(x,v)
:= \left(
x + t v - \frac{t^{1-\alpha}-1}{1-\alpha}\,\mathbf{A}_\infty(v),\,
v
\right).
\]
Here, \( \mathbf{A}_\infty(v) \) is a velocity correction defined by
\begin{equation}\label{eq: A(v) definition-0}
\mathbf{A}_\infty(v)
:=
- \iint_{\mathbb{R}^6}
\nabla w\big(v - \mathcal{V}^+(\tilde{x},\tilde{v})\big)\,
f_0(\tilde{x},\tilde{v})\,
d\tilde{x}\, d\tilde{v},
\end{equation}
where the momentum limit \( \mathcal{V}^+(x,v) \) is given by
\[
\mathcal{V}^+(x,v)
:=
\lim_{t \to \infty} \mathcal{V}(t,0,x,v)
=
v + \int_0^\infty
\mathbf{E}_f\big(t,\mathcal{X}(\tau,0,x,v)\big)\, d\tau,
\]
established by Lemma~\ref{lemma: properties of the momentum limit}. 
The rigorous construction of \( \Phi^{\mathrm{ref}}(t) \) is given in Section~\ref{sec: construction of the modified reference flow}. 
We say that \( f(t) \) exhibits \emph{modified scattering} as \( t \to \infty \) if there exists an asymptotic profile
\[
f^+(x,v):=\lim_{t\to\infty} f\big(t, \Phi^{\mathrm{ref}}(t)(x,v)\big)
\]
where this convergence occurs in a suitable norm.

In this direction, our main result establishes modified scattering for small-data global solutions to the Vlasov--Riesz system in \( L_{x,v}^\infty(\mathbb{R}^6) \).

\begin{theorem}[Modified scattering]\label{theorem: modified scattering}
For \( \frac{1}{2} < \alpha < 1 \), there exists a small constant \( \eta_* > 0 \) such that the following holds. Suppose that \( f_0 \ge 0 \) and 
\begin{equation}\label{eq: small data condition}
\| \langle x \rangle f_0 \|_{W_{x,v}^{1,1}(\mathbb{R}^6)\cap W_{x,v}^{1,\infty}(\mathbb{R}^6)}+ \|\langle x\rangle f_0 \|_{W_x^{1,1}(\mathbb{R}^3; W_v^{1,\infty}(\mathbb{R}^3))\cap W_v^{1,1}(\mathbb{R}^3; W_x^{1,\infty}(\mathbb{R}^3))}
\le \eta_*,
\end{equation}
and let \(f(t) \in C([0,\infty); W^{1,1}_{x,v}(\mathbb{R}^6) \cap W^{1,\infty}_{x,v}(\mathbb{R}^6))\) be the global solution to the Vlasov--Riesz system \eqref{eq: VR} with initial data \( f_0 \) such that \eqref{eq: dispersion estimates-0} holds. Then, there exists an asymptotic profile \( f^+ \in L^1_{x,v}(\mathbb{R}^6) \cap L^\infty_{x,v}(\mathbb{R}^6) \) such that for all \( t \ge 1 \),
\begin{equation}\label{eq: modified scattering}
\big\| f \big( t, \Phi^{\mathrm{ref}}(t)(x,v) \big) - f^+(x,v) \big\|_{L^\infty_{x,v}(\mathbb{R}^6)}
\lesssim \frac{\eta_*^2}{t^{2\alpha - 1}}.
\end{equation}
\end{theorem}

\begin{remark}[Modified scattering in the strictly long-range regime] $(i)$ In the recent work~\cite{HuangKwon}, Huang and Kwon established modified scattering for the Vlasov--Riesz system in the range \( 1-\delta \le \alpha < 1 \) with small \( \delta > 0 \)\footnote{By numerical calculations, the authors obtained \( \delta = 0.16 \).}. 
The main result of this paper extends this range to \( \frac{1}{2} < \alpha < 1 \) with a more robust, Lagrangian argument that emphasizes the convergence of modified wave operators. In particular, we answer the open question that they pose within Remark 3.5(3) of \cite{HuangKwon}.\\
$(ii)$ The regime \( 0 < \alpha \le \frac{1}{2} \) is not covered by the present approach and is left for future work. 
Treating this case would likely require higher-order corrections to the modified wave operators constructed in Section~\ref{sec: Finite-time modified wave operator}, and appears to depend on additional decay estimates for derivatives of the spatial density, which are currently unavailable.
\end{remark}

\begin{remark}[Quantum mechanical analogy] 
By the quantum--classical correspondence, the Vlasov equation \eqref{eq: VR} is closely related to the nonlinear Hartree equation (NLH) with the same inverse power-law potential:
\[
i\partial_t u = -\frac{1}{2}\Delta u + \Bigg(\frac{\lambda}{|x|^\alpha} * |u|^2 \Bigg) u,
\]
where \( u = u(t,x): [0,\infty) \times \mathbb{R}^3 \to \mathbb{C} \).
These two models share several structural features.
In fact, the long-time asymptotic behavior of NLH has been extensively studied 
\cite{HayashiNaumkin1998-1, HayashiNaumkin1998-2, HayashiNaumkin1998-3, HayashiNaumkin2001, Hirata1995}, with a substantial body of work developed prior to its classical counterpart.
In particular, Hayashi and Naumkin \cite{HayashiNaumkin2001} established small-data global well-posedness for \( 0 < \alpha < 1 \) \cite[Theorem 1.1]{HayashiNaumkin2001} and proved modified scattering for \( \frac{1}{2} < \alpha < 1 \) \cite[Theorem 1.2]{HayashiNaumkin2001}, corresponding to our Theorem \ref{thm: global well-posedness} and Theorem \ref{theorem: modified scattering}, respectively, in this article. 
Furthermore, the rate of convergence \eqref{eq: modified scattering} obtained in our main result coincides with that for the quantum modified scattering. 
To the best of the authors' knowledge, for \( 0 < \alpha \le \frac{1}{2} \), modified scattering is currently unknown for the NLH equation.
Hence, establishing modified scattering for the Vlasov equation \eqref{eq: VR} in this regime may also be challenging, or even impossible.
\end{remark}

\begin{remark}[Initial data condition] $(i)$ To construct global solutions with small initial data satisfying the dispersive estimate \eqref{eq: dispersion estimates-0}, it suffices to assume the smallness condition \eqref{eq: small data condition for decay only}. 
The additional weighted norm bound
\(\| \langle x \rangle f_0 \|_{W_{x,v}^{1,\infty}(\mathbb{R}^6)}\)
imposed in \eqref{eq: small data condition} is needed only to establish the modified scattering result \eqref{eq: modified scattering}. 
Whether this weighted bound is strictly necessary remains an open question. 
In our proof, it is mainly used to control the term \(\frac{x - \tilde{x}}{t}\), which naturally arises in our approach, for instance, as in \eqref{eq: asymptotic with correction}. 
Similar weighted norm conditions are also used in the literature on the nonlinear Hartree equation to obtain modified scattering \cite[Theorem 1.2]{HayashiNaumkin2001}. \\
$(ii)$ It is desirable to establish the convergence estimate \eqref{eq: modified scattering} in \( L^1_{x,v}(\mathbb{R}^6) \). 
Based on our previous work \cite{HongPankavich}, this would require the modified wave operator \( \mathcal{W}^{\textup{mod}}(t) \), introduced in Section~\ref{sec: Finite-time modified wave operator}, to be sufficiently close to the identity in a stronger sense, namely
\[
\| \nabla_{(x,v)} \mathcal{W}^{\textup{mod}}(t)(x,v) - \mathbb{I}_3 \| \lesssim \eta_*.
\] 
However, such a bound is currently out of reach due to the term \( \frac{x-\tilde{x}}{t} \) in \eqref{eq: asymptotic with correction}. 
Herein, we instead employ a weaker version that yields convergence in \( L_{x,v}^\infty(\mathbb{R}^6) \) via Proposition \ref{prop: modified wave operator} (2).
\end{remark}

The asymptotic dynamics of Vlasov-type equations have been studied using two complementary approaches, which are often combined in the literature. 
The first approach is based on the analysis of characteristic flows, as in \eqref{eq: forward characteristic ODE for VP-0}, and primarily relies on ODE techniques. 
This method was introduced in the seminal work of Bardos and Degond \cite{BD}, where global well-posedness for small initial data and decay estimates for the Vlasov--Poisson system were established. The second approach emphasizes the PDE structure of \eqref{eq: VR}, employing dispersive and energy-type estimates. 
This framework was developed more recently by Ionescu, Pausader, Wang, and Widmayer \cite{IPWW22} and applied in various contexts, including the work of Huang and Kwon \cite{HuangKwon}. 
Many authors combine elements from both approaches to effectively analyze the long-time dynamics of the system.

In this article, we primarily follow the approach of Bardos and Degond \cite{BD}, while incorporating an additional ingredient: the finite- and infinite-time modified wave operators
\[
\mathcal{W}^{\textup{mod}}(t),\ \mathcal{W}^{\textup{mod},+}:\mathbb{R}^3\times\mathbb{R}^3 \to \mathbb{R}^3\times\mathbb{R}^3
\]
for the characteristic flows on the phase space \(\mathbb{R}^3\times\mathbb{R}^3\). See \eqref{eq: finite-time modified wave operator} and \eqref{eq: infinite-time modified wave operator} for their precise definition. 
Motivated by their quantum-mechanical analogues \cite{Yajima}, these wave operators were introduced in our previous work \cite{HongPankavich}, where the nonrelativistic limit of scattering states was studied in the short-range case. 
In the present work, the operators are modified so that the reference flow admits proper asymptotic limits (see Section~\ref{sec: construction of the modified reference flow}). From a mathematical perspective, the wave operator formulation for characteristic flows is equivalent to the original formulation, but it provides a more convenient framework to describe the asymptotic dynamics. 
In particular, this formulation shows that the position and momentum components of the modified wave operators \(\mathcal{W}^{\textup{mod}}(t)(x,v)\) converge at different rates (see Proposition~\ref{prop: modified wave operator}). 
This motivates establishing separate bounds on the derivatives \(\nabla_x g(t,x,v)\) and \(\nabla_v g(t,x,v)\) of the modified distribution \(g(t,x,v)\) in Proposition~\ref{prop: global bounds for modified distribution}, together with an appropriate weighted norm bound. 
These bounds are obtained using PDE estimates. Finally, by combining the convergence of the modified wave operator \(\mathcal{W}^{\textup{mod}}(t)(x,v) \to \mathcal{W}^{\textup{mod},+}(x,v)\) with the bounds on \(g(t)\), we are able to complete the proof of modified scattering.

\subsection{Organization of the paper}
The remainder of the paper is organized as follows. 
In Section~\ref{sec: Global well-posedness with dispersion bounds}, we present the preliminary small-data theory for the Vlasov--Riesz system, including global well-posedness, the analysis of the associated characteristic flow, and dispersion bounds (Theorem~\ref{thm: global well-posedness}). 
Section~\ref{sec: Construction of the modified wave operator} is devoted to the construction of the finite- and infinite-time modified wave operators for the characteristic flows, together with their convergence estimates (Proposition~\ref{prop: modified wave operator}). 
In Section~\ref{sec: Energy estimates for modified distributions}, we establish weighted $W^{1,\infty}$ bounds for the modified distribution, including bounds on derivatives and weighted norms (Proposition~\ref{prop: global bounds for modified distribution}). 
Finally, Section~\ref{sec: proof of modified scattering} combines these results to complete the proof of modified scattering.

\subsection{Notation}
For \( k \in \mathbb{N}_0 \) and \( r \in [1, \infty] \), let \( W^{k,r}(\mathbb{R}^d) \) denote the standard Sobolev space with norm
\[
\|u\|_{W^{k,r}(\mathbb{R}^d)} := \sum_{|m| \le k} \|\nabla^m u\|_{L^r(\mathbb{R}^d)},
\]
where \( m = (m_1, \dots, m_d) \) is a multi-index with \( |m| = m_1 + \cdots + m_d \). 
We define the mixed-norm space \( W_x^{k, r}(\mathbb{R}^3; W_v^{k', r'}(\mathbb{R}^3)) \) with norm
\begin{equation}\label{eq: mixed norm definition}
\|f\|_{W_x^{k,r}(\mathbb{R}^3; W_v^{k', r'}(\mathbb{R}^3))} := \sum_{|m| \le k, |m'| \le k'} \|\nabla_x^{m} \nabla_v^{m'} f\|_{L_x^{r}(\mathbb{R}^3; L_v^{r'}(\mathbb{R}^3))},
\end{equation}
which can be written explicitly as
\[
\sum_{|m| \le k, |m'| \le k'} \left\{ \int_{\mathbb{R}^3} \left( \int_{\mathbb{R}^3} \big| \nabla_x^{m} \nabla_v^{m'} f \big|^{r'} \, dv \right)^{\frac{r}{r'}} dx \right\}^{\frac{1}{r}}.
\]
Throughout this article, we abbreviate \( W_x^{k,r}(\mathbb{R}^3; W_v^{k', r'}(\mathbb{R}^3)) \) as \( W_x^{k, r} W_v^{k', r'} \) when there is no ambiguity. 
Similarly, interchanging the roles of \( x \) and \( v \), we write \(W_v^{k, r} W_x^{k', r'} := W_v^{k, r}(\mathbb{R}^3; W_x^{k', r'}(\mathbb{R}^3))\).

\subsection{Acknowledgement}
Y. Hong was supported by the National Research Foundation of Korea (NRF) grant funded by the Korean government (MSIT) (No. RS-2023-00219980 and RS-2026-25479401).
S. Pankavich was supported by the US National Science Foundation under award DMS-2107938.

\section{Global well-posedness with dispersion bounds}\label{sec: Global well-posedness with dispersion bounds}

In this section, we show that if the initial data \( f_0 \) is sufficiently small in an appropriate norm, then the Vlasov--Riesz system \eqref{eq: VR} admits a unique global-in-time solution, and the associated force field satisfies decay estimates.

\begin{theorem}[Global well-posedness for small data solutions]\label{thm: global well-posedness}
For \( 0 < \alpha < 1 \), there exists a small constant \( \eta_* > 0 \) such that the following holds. If \( f_0 \geq 0 \) and
\begin{equation}\label{eq: smallness condition for GWP}
\| f_0 \|_{W_{x,v}^{1,1}(\mathbb{R}^6)\cap W_{x,v}^{1,\infty}(\mathbb{R}^6)} + \| f_0 \|_{W_x^{1,1}(\mathbb{R}^3; W_v^{1,\infty}(\mathbb{R}^3))\cap W_v^{1,1}(\mathbb{R}^3; W_x^{1,\infty}(\mathbb{R}^3))} \leq \eta_*,
\end{equation}
then the Vlasov--Riesz system \eqref{eq: VR} admits a unique global solution \( f(t)\in C_t([0,\infty);L_{x,v}^\infty(\mathbb{R}^6)) \) with initial data \( f_0 \) such that
\begin{equation}\label{eq: dispersion estimates-density}
\sup_{t \geq 0}
\Big(
\langle t \rangle^{3}
\| \rho_{f(t)} \|_{W_x^{1,\infty}(\mathbb{R}^3)}
\Big)
\lesssim \eta_*.
\end{equation}
As a consequence, we have
\begin{equation}\label{eq: dispersion estimates}
\sup_{t \geq 0}
\bigg(
\langle t \rangle^{1+\alpha}
\| \mathbf{E}_f(t) \|_{L_x^{\infty}(\mathbb{R}^3)} +
\langle t \rangle^{2+\alpha}
\|\nabla\mathbf{E}_f(t) \|_{L_x^{\infty}(\mathbb{R}^3)} +
\langle t \rangle^{3}
\| \nabla^2 \mathbf{E}_f(t) \|_{L_x^\infty(\mathbb{R}^3)}
\bigg)
\lesssim \eta_*.
\end{equation}
\end{theorem}

\subsection{Proof of global well-posedness}

We follow the argument of Bardos--Degond~\cite{BD}. In particular, the two lemmas below provide key estimates for the contraction mapping argument.

\begin{lemma}[Dispersion bounds imply characteristic flow estimates]\label{lemma: dispersion bounds imply characteristic flow estimates}
Let  $0 < \alpha < 1$. Suppose that a vector field \(\mathbf{E} = \mathbf{E}(t,x) : [0,\infty) \times \mathbb{R}^3 \to \mathbb{R}^3\) satisfies the global-in-time decay bounds
\begin{equation}\label{eq: E uniform decay bounds}
\sup_{t \ge 0}
\Big(
\langle t \rangle^{1+\alpha}
\| \mathbf{E}(t) \|_{L_x^\infty(\mathbb{R}^3)}
+
\langle t \rangle^{2+\alpha}
\| \nabla \mathbf{E}(t) \|_{L_x^\infty(\mathbb{R}^3)}
\Big)\leq C_0.
\end{equation}
Then, for each $t>0$ and $(x,v)\in\mathbb{R}^3\times\mathbb{R}^3$, the backward-in-time Hamiltonian ODE system
\begin{equation}\label{eq: Hamiltonian ODE0}
\left\{
\begin{aligned}
\partial_s \Big(\mathcal{X}(s,t,x,v), \mathcal{V}(s,t,x,v)\Big) &= \Big(\mathcal{V}(s,t,x,v), \mathbf{E}\big(s, \mathcal{X}(s,t,x,v)\big)\Big), \\
\big(\mathcal{X}(t,t,x,v),\mathcal{V}(t,t,x,v)\big) &= (x,v)
\end{aligned}
\right.
\end{equation}
admits a unique solution $(\mathcal{X}(s,t,x,v), \mathcal{V}(s,t,x,v)): [0,t]\to\mathbb{R}^3\times\mathbb{R}^3$ such that
\begin{equation}\label{eq: derivative characteristic bound}
\left\{
\begin{aligned}
\| \nabla_x \mathcal{X}(s,t,x,v) - \mathbb{I}_3 \|
+ \| \nabla_x \mathcal{V}(s,t,x,v) \|
&\leq KC_0, \\[0.3em]
\big\| \nabla_v \mathcal{X}(s,t,x,v) + (t-s)\mathbb{I}_3 \big\|
+ \| \nabla_v \mathcal{V}(s,t,x,v) - \mathbb{I}_3 \|
&\leq KC_0(t-s)
\end{aligned}
\right.
\end{equation}
for some $K > 0$.
\end{lemma}

\begin{proof}
Throughout the proof, we fix $t > 0$ and $(x,v) \in \mathbb{R}^3 \times \mathbb{R}^3$.
For $0 \le s \le t$, we set
\[
\big(\tilde{\mathcal{X}}_s, \tilde{\mathcal{V}}_s\big)
:= \big(\mathcal{X}(s,t,x,v), \mathcal{V}(s,t,x,v)\big).
\]
From \eqref{eq: Hamiltonian ODE0}, the characteristic curves satisfy the integral
representations
\[
(\tilde{\mathcal{X}}_s, \tilde{\mathcal{V}}_s)
=\bigg( x - (t-s)v
+ \int_s^t (s_1 - s)\,
\mathbf{E}(s_1, \tilde{\mathcal{X}}_{s_1}) \, ds_1, v
- \int_s^t
\mathbf{E}(s_1, \tilde{\mathcal{X}}_{s_1}) \, ds_1\bigg).
\]
Differentiating with respect to $x$ and $v$, we obtain
\[
\left\{
\begin{aligned}
\nabla_{(x,v)} \tilde{\mathcal{X}}_s
&=\big[\mathbb{I}_3\ \  -(t-s)\mathbb{I}_3\big]^\top+ \int_s^t (s_1 - s)\,
\nabla \mathbf{E}(s_1, \tilde{\mathcal{X}}_{s_1})
\nabla_{(x,v)} \tilde{\mathcal{X}}_{s_1} \, ds_1, \\
\nabla_{(x,v)} \tilde{\mathcal{V}}_s&= [0\ \  \mathbb{I}_3]^\top- \int_s^t
\nabla \mathbf{E}(s_1, \tilde{\mathcal{X}}_{s_1})
\nabla_{(x,v)} \tilde{\mathcal{X}}_{s_1} \, ds_1.
\end{aligned}
\right.
\]
Applying the decay bound \eqref{eq: E uniform decay bounds}, we obtain 
\[
\begin{aligned}
&\| \nabla_x \tilde{\mathcal{X}}_s - \mathbb{I}_3 \|+\bigg\|\frac{\nabla_v \tilde{\mathcal{X}}_s}{t-s} +\mathbb{I}_3 \bigg\|+\| \nabla_x \tilde{\mathcal{V}}_s \|+\frac{\|\nabla_v \tilde{\mathcal{V}}_s-\mathbb{I}_3\|}{t-s}\\
&\lesssim\int_s^t 
\frac{C_0}
{\langle s_1 \rangle^{1+\alpha}}\bigg(\| \nabla_x \tilde{\mathcal{X}}_{s_1}\|+\bigg\|\frac{\nabla_v \tilde{\mathcal{X}}_{s_1}}{t-s_1} \bigg\|\bigg ) \, ds_1\\
&\lesssim
 \int_s^t \frac{C_0}{\langle s_1 \rangle^{1+\alpha}}\,ds_1
+  \int_s^t
\frac{C_0}
{\langle s_1 \rangle^{1+\alpha}}\bigg\{\| \nabla_x \tilde{\mathcal{X}}_{s_1} - \mathbb{I}_3 \|+\bigg\|\frac{\nabla_v \tilde{\mathcal{X}}_{s_1}}{t-s_1} +\mathbb{I}_3 \bigg\|\bigg\} \, ds_1,
\end{aligned}
\]
which yields \eqref{eq: derivative characteristic bound} by Gr\"onwall's inequality as $\alpha > 0$.
\end{proof}

Let 
\[
\Phi(t) = \big(\Phi_1(t), \Phi_2(t)\big) : [0,\infty) \times \mathbb{R}^3 \times \mathbb{R}^3 \to \mathbb{R}^3 \times \mathbb{R}^3
\]
denote the forward-in-time characteristic flow associated with \(\mathbf{E}\), defined as the solution to the ODE system
\begin{equation}\label{eq: forward characteristic ODE}
\left\{
\begin{aligned}
\partial_t \Phi(t)(x,v) &=\Big(\Phi_2(t)(x,v), \mathbf{E}\big(t, \Phi_1(t)(x,v)\big)\Big), \\
\Phi(0)(x,v) &= (x,v)
\end{aligned}
\right.
\end{equation}
for each \((x,v) \in \mathbb{R}^6\). 
Then, the pushforward of \(f_0\) by \(\Phi(t)\) can be expressed as
\[
\big(\Phi(t)_\# f_0\big)(x,v) = f_0\big(\Phi(t)^{-1}(x,v)\big)
= f_0\big(\mathcal{X}(0,t,x,v), \mathcal{V}(0,t,x,v)\big).
\]

\begin{remark}\label{remark: PDE ODE relation}
The pushforward measure \(f(t) = \Phi(t)_\# f_0 = f_0(\Phi(t)^{-1}(\cdot,\cdot))\) is a weak solution to the Vlasov equation. Indeed, for any test function
\(\varphi \in \mathcal{S}(\mathbb{R}^3 \times \mathbb{R}^3)\),
\[
\begin{aligned}
\frac{d}{dt} \langle \varphi, f(t) \rangle
&= \frac{d}{dt} \big\langle \varphi\big(\Phi(t)(x,v)\big), f_0(x,v) \big\rangle \\
&= \Big\langle \nabla_x \varphi \big(\Phi(t)(x,v)\big) \cdot \Phi_2(t)(x,v)+ \nabla_v \varphi \big(\Phi(t)(x,v)\big) \cdot \mathbf{E}(t, \Phi_1(t)(x,v)), f_0(x,v) \Big\rangle \\
&= - \big\langle \varphi, \nabla_{(x,v)} \cdot \big(v f(t), \mathbf{E}(t)\big) \Phi(t)_\# f_0 \big\rangle = - \big\langle \varphi, v \cdot \nabla_x f(t) + \mathbf{E}(t) \cdot \nabla_v f(t) \big\rangle,
\end{aligned}
\]
where \(\langle \cdot, \cdot \rangle = \langle \cdot, \cdot \rangle_{L_{x,v}^2(\mathbb{R}^3 \times \mathbb{R}^3)}\).
\end{remark}

The following lemma shows that under the decay assumption \eqref{eq: E uniform decay bounds}, the pushforward measure satisfies dispersion bounds.

\begin{lemma}[Characteristic flow bounds imply dispersion estimates]\label{lemma: characteristic flow bounds imply dispersion estimates}
If \(\Xi(s,t,x,v)=(\mathcal{X}(s,t,x,v), \mathcal{V}(s,t,x,v))\) satisfies the derivative bounds
\begin{equation}\label{eq: derivative bound for density estimate}
\left\{\begin{aligned}
\| \nabla_x \mathcal{X}(0,t,x,v) - \mathbb{I}_3 \|
+ \| \nabla_x \mathcal{V}(0,t,x,v) \|
&\leq \frac{1}{10}, \\[0.3em]
\big\| \nabla_v \mathcal{X}(0,t,x,v) + t\mathbb{I}_3 \big\|
+ \| \nabla_v \mathcal{V}(0,t,x,v) - \mathbb{I}_3 \|
&\leq \frac{t}{10},
\end{aligned}\right.
\end{equation}
then the density of the pushforward via the associated Hamiltonian flow $\Phi(t)$ satisfies
\begin{align*}
 \sup_{t\geq0} & \left (\|\rho_{\Phi(t)_\# f_0} \|_{W_x^{1,1}(\mathbb{R}^3)} + \langle t \rangle^3\| \rho_{\Phi(t)_\# f_0} \|_{W_x^{1,\infty}(\mathbb{R}^3)} \right )\\
    & \leq 4 
\| f_0 \|_{W_{x,v}^{1,1} \cap W_x^{1,1}(\mathbb{R}^3; W_v^{1,\infty}(\mathbb{R}^3))\cap W_v^{1,1}(\mathbb{R}^3; W_x^{1,\infty}(\mathbb{R}^3))}.
\end{align*}
\end{lemma}

\begin{proof}
From the integral representation of the pushforward density, we have
\begin{equation}\label{eq: density of pushforward measure}
\rho_{\Phi(t)_\# f_0}(x) = \int_{\mathbb{R}^3} f_0\big(\Xi(0,t,x,v)\big) \, dv,
\end{equation}
where $\Xi(0,t,x,v)=(\mathcal{X}(0,t,x,v),\mathcal{V}(0,t,x,v))$. 
For $t \ge 1$, we change variables using \(v\mapsto y = \mathcal{X}(0,t,x,v)\). Then, it follows from \eqref{eq: derivative bound for density estimate} that
\[
\begin{aligned}
\rho_{\Phi(t)_\# f_0}(x) 
&= \int_{\mathbb{R}^3} f_0\big(y, \mathcal{V}(0,t,x,v(y))\big) 
   \, \Big| \det\Big\{\nabla_v \mathcal{X}\big(0,t,x,v(y)\big)\Big\} \Big|^{-1} \, dy 
\leq \frac{2}{t^3} \, \| f_0 \|_{L_x^1 L_v^\infty}.
\end{aligned}
\]
For $0 \le t \le 1$, we instead change variables \(y = \mathcal{V}(0,t,x,v)\). Then, similarly, by \eqref{eq: derivative characteristic bound} we obtain
\[
\rho_{\Phi(t)_\# f_0}(x) 
\leq 2\| f_0 \|_{L_v^1 L_x^\infty}.
\]
Combining the two cases, we deduce
\begin{equation}\label{eq: characteristic flow bounds imply dispersion estimates, 0 derivative}
\|\rho_{\Phi(t)_\# f_0}\|_{L_x^\infty} \leq \frac{2}{\langle t\rangle^3} \, \| f_0 \|_{L_x^1 L_v^\infty\cap L_v^1 L_x^\infty}.
\end{equation}
Next, differentiating \eqref{eq: density of pushforward measure} and applying \eqref{eq: derivative bound for density estimate} and \eqref{eq: characteristic flow bounds imply dispersion estimates, 0 derivative}, we obtain
$$\begin{aligned}
\|\nabla_x \rho_{\Phi(t)_\# f_0}\|_{L_x^\infty}&=\bigg\|\int_{\mathbb{R}^3} (\nabla_{(x,v)} f_0)\big(\Xi(0,t,x,v)\big) 
\cdot \nabla_x \Xi(0,t,x,v) \, dv\bigg\|_{L_x^\infty}\\
&\leq2 
\bigg\|\int_{\mathbb{R}^3} \big|\nabla_{(x,v)} f_0\big(\Xi(0,t,x,v)\big)\big| \, dv\bigg\|_{L_x^\infty} 
= 2\big\|\rho_{\Phi(t)_\#|\nabla_{(x,v)} f_0|}\big\|_{L_x^\infty}\\
&\leq 
\frac{4}{\langle t \rangle^3} \, \| \nabla_{(x,v)} f_0 \|_{L_x^1 L_v^\infty \cap L_v^1 L_x^\infty}.
\end{aligned}$$
Combining this with \eqref{eq: characteristic flow bounds imply dispersion estimates, 0 derivative} yields the stated $W^{1,\infty}$ estimate.
Finally, we obtain the corresponding $W^{1,1}$ estimate using mass conservation and the measure preserving property of the characteristics. In particular, proceeding similarly for the $L^1$ derivative estimate as the $L^\infty$ derivative estimate above, we find
$$\|\nabla_x \rho_{\Phi(t)_\# f_0}\|_{L_x^1} \leq 2\big\|\rho_{\Phi(t)_\#|\nabla_{(x,v)} f_0|}\big\|_{L_x^1} \leq 4\| \nabla_{(x,v)} f_0 \|_{L_{x,v}^1},$$
and the proof is complete.
\end{proof}

\begin{lemma}[Interpolation inequality]\label{lemma: interpolation inequality}
Let $0<\alpha<1$ and $m=1,2$. Then, for any nonnegative density $\rho:\mathbb{R}^3\to[0,\infty)$,
\[
\int_{\mathbb{R}^3} \frac{\rho(y)}{|x-y|^{m+\alpha}} \, dy \;\leq C_1
\|\rho\|_{L_x^1(\mathbb{R}^3)}^{\frac{3-m-\alpha}{3}} 
\|\rho\|_{L_x^\infty(\mathbb{R}^3)}^{\frac{m+\alpha}{3}}.
\]
\end{lemma}

\begin{proof}
The proof is classical. Fix $x\in\mathbb{R}^3$ and decompose the integral into near-field and far-field contributions with a radius parameter $R>0$ so that
\[
\int_{\mathbb{R}^3} \frac{\rho(y)}{|x-y|^{m+\alpha}} \, dy
= \int_{|x-y|\le R} \frac{\rho(y)}{|x-y|^{m+\alpha}} \, dy
+ \int_{|x-y|\ge R} \frac{\rho(y)}{|x-y|^{m+\alpha}} \, dy.
\]
For the near-field term, using $\rho(y)\le \|\rho\|_{L^\infty}$, we obtain
\[
\int_{|x-y|\le R} \frac{\rho(y)}{|x-y|^{m+\alpha}} \, dy 
\lesssim \|\rho\|_{L^\infty} \int_0^R r^{2-(m+\alpha)} \, dr 
\sim R^{3-m-\alpha} \|\rho\|_{L^\infty}.
\]
For the far-field term, using $\rho \in L^1$, we have
\[
\int_{|x-y|\ge R} \frac{\rho(y)}{|x-y|^{m+\alpha}} \, dy 
\le \frac{1}{R^{m+\alpha}} \int_{\mathbb{R}^3} \rho(y) \, dy
= \frac{\|\rho\|_{L^1}}{R^{m+\alpha}}.
\]
Optimizing in $R$ by choosing \(R = \big( \frac{\|\rho\|_{L^1}}{\|\rho\|_{L^\infty}}\big)^{1/3}\), we balance the two contributions and obtain the claimed interpolation bound.
\end{proof}

 We are now ready to prove the global well-posedness of the Vlasov--Riesz system \eqref{eq: VR}.

\begin{proof}[Proof of Theorem \ref{thm: global well-posedness}]
We fix a small initial data $f_0$ satisfying the smallness condition \eqref{eq: smallness condition for GWP}, and construct sequences $\{f^{(n)}\}_{n=1}^\infty$, $\{\mathbf{E}^{(n)}\}_{n=1}^\infty$ and $\{\Xi^{(n)}(s,t,x,v)\}_{n=0}^\infty$, with $\Xi^{(n)}(s,t,x,v)=(\mathcal{X}^{(n)}(s,t,x,v), \mathcal{V}^{(n)}(s,t,x,v))$ for $0\leq s<t<\infty$ as follows. Let $\Xi^{(0)}(s,t,x,v)=(x-(t-s)v, v)$ be the solution to the ODE \eqref{eq: Hamiltonian ODE0} with $\mathbf{E}\equiv0$. For each integer $n\geq 1$, we iteratively construct
$$f^{(n)}(t,x,v):=f_0\big(\Xi^{(n-1)}(0,t,x,v)\big)\quad\textup{and}\quad\mathbf{E}^{(n)}(t,x):=-\nabla( w*\rho_{f^{(n)}})(t,x).$$
Then, we let $\Xi^{(n)}(s,t,x,v)$ be the solution to the ODE \eqref{eq: Hamiltonian ODE0} with $\mathbf{E}(t,x)=\mathbf{E}^{(n)}(t,x)$.

We claim that for all $n \ge 1$,
\begin{equation}\label{eq: sequence density estimate}
\sup_{t \ge 0}
\Big(
\| \rho_{f^{(n)}(t)} \|_{W_x^{1,1}} +
\langle t \rangle^{3} \| \rho_{f^{(n)}(t)} \|_{W_x^{1,\infty}}
\Big)
\le 4 \eta_*,
\end{equation}
which, by Lemma~\ref{lemma: interpolation inequality}, implies
\begin{equation}\label{eq: force field density estimate}
\sup_{t \ge 0} \Big(
\langle t \rangle^{1+\alpha} \| \mathbf{E}^{(n)}(t) \|_{L_x^\infty} +
\langle t \rangle^{2+\alpha} \| \nabla \mathbf{E}^{(n)}(t) \|_{L_x^\infty}
\Big) \leq C \eta_*.
\end{equation}
Indeed, by Lemma~\ref{lemma: characteristic flow bounds imply dispersion estimates}, \eqref{eq: sequence density estimate} holds for $n=1$. For the induction step, assume that \eqref{eq: sequence density estimate} holds for $n=k \ge 1$. Then, by Lemma~\ref{lemma: interpolation inequality}, $\mathbf{E}^{(k)}$ satisfies the assumption \eqref{eq: E uniform decay bounds} in Lemma~\ref{lemma: dispersion bounds imply characteristic flow estimates} with $C_0 \sim \eta_*$. Subsequently, Lemma~\ref{lemma: dispersion bounds imply characteristic flow estimates} implies that $\Xi^{(k+1)}(s,t,x,v)$ satisfies the assumption in Lemma~\ref{lemma: characteristic flow bounds imply dispersion estimates}. Therefore, by Lemma~\ref{lemma: characteristic flow bounds imply dispersion estimates}, \eqref{eq: sequence density estimate} holds for $n = k+1$.

Next, using the integral representation for  $\Xi^{(n)}(s,t,x,v)=\tilde{\Xi}_s^{(n)}=(\tilde{\mathcal{X}}_s^{(n)},\tilde{\mathcal{V}}_s^{(n)})$, we write
$$\tilde{\Xi}_s^{(n)}=\bigg(x-(t-s)v+\int_s^t (s_1-s)\mathbf{E}^{(n)}\big(s_1,\tilde{\mathcal{X}}_{s_1}^{(n)}\big)ds_1,v-\int_s^t \mathbf{E}^{(n)}\big(s_1,\tilde{\mathcal{X}}_{s_1}^{(n)}\big)ds_1\bigg).$$
Then, we obtain 
$$\begin{aligned}
\mathcal{D}_{n}:&=\sup_{0\leq s\leq t\leq 1}\big\|\Xi^{(n+1)}(s,t,x,v)-\Xi^{(n)}(s,t,x,v)\big\|_{C_{x,v}}\\
&\leq \int_0^1 \langle s_1\rangle\Big(\|(\mathbf{E}^{(n+1)}-\mathbf{E}^{(n)})(s_1)\|_{C_{x}}+\|\nabla\mathbf{E}^{(n)}(s_1)\|_{C_x}\mathcal{D}_n\Big)ds_1.
\end{aligned}$$
Hence, by \eqref{eq: force field density estimate} with small $\eta_*>0$, it follows that
\begin{equation}\label{eq: delta_T^n bound}
\mathcal{D}_{n}\lesssim \int_0^1 \langle s_1\rangle\|(\mathbf{E}^{(n+1)}-\mathbf{E}^{(n)})(s_1)\|_{C_{x}}ds_1.
\end{equation}
Note that by construction we have
$$\begin{aligned}
&(\mathbf{E}^{(n+1)}-\mathbf{E}^{(n)})(s_1,y)\\
&=\iint_{\mathbb{R}^6} \nabla w(y-x) \Big\{f_0\big(\Xi^{(n)}(0,s_1,x,v)\big)-f_0\big(\Xi^{(n-1)}(0,s_1,x,v)\big)\Big\}dxdv\\
&=\int_0^1\iint_{\mathbb{R}^6} \nabla w(y-x) (\nabla_{(x,v)}f_0)\big(\Xi_\theta^{(n)}(0,s_1,x,v)\big)\\
&\qquad\qquad\qquad\cdot\big(\Xi^{(n)}(0,s_1,x,v)-\Xi^{(n-1)}(0,s_1,x,v)\big)dxdvd\theta,
\end{aligned}$$
where
$$\Xi_\theta^{(n)}(s,t,x,v):=\theta\Xi^{(n)}(s,t,x,v)+(1-\theta)\Xi^{(n-1)}(s,t,x,v),$$
and subsequently, by Lemma \ref{lemma: interpolation inequality},
$$\begin{aligned}
&\big|(\mathbf{E}^{(n+1)}-\mathbf{E}^{(n)})(s_1,y)\big|\\
&\leq\mathcal{D}_{n-1}\int_0^1\iint_{\mathbb{R}^6} |\nabla w(y-x)| \big|(\nabla_{(x,v)}f_0)\big(\Xi_\theta^{(n)}(0,s_1,x,v)\big)\big|dxdvd\theta\\
&=\mathcal{D}_{n-1}\int_0^1\int_{\mathbb{R}^3} |\nabla w(y-x)| \rho_{|(\nabla_{(x,v)}f_0)(\Xi_\theta^{(n)}(0,s_1,x,v))|}dxd\theta\\
&\lesssim \mathcal{D}_{n-1}\int_0^1\|\rho_{|(\nabla_{(x,v)}f_0)(\Xi_\theta^{(n)}(0,s_1,x,v))|}\|_{L_x^1}^{\frac{2-\alpha}{3}}\|\rho_{|(\nabla_{(x,v)}f_0)(\Xi_\theta^{(n)}(0,s_1,x,v))|}\|_{L_x^\infty}^{\frac{1+\alpha}{3}}d\theta
\end{aligned}$$
for every $y \in \mathbb{R}^3$ and $0\leq s_1\leq 1$. However, since $\Xi^{(n)}(s,t,x,v)$ obeys \eqref{eq: derivative characteristic bound}  with small $C_0\sim\eta_*$, so does $\Xi_\theta^{(n)}(s,t,x,v)$. Hence, by Lemma \ref{lemma: characteristic flow bounds imply dispersion estimates}, we can show that for $0\leq s_1\leq 1$, 
$$\big\Vert(\mathbf{E}^{(n+1)}-\mathbf{E}^{(n)})(s_1)\big\Vert_{L^\infty_x} \lesssim\frac{\eta_*}{\langle s_1\rangle^{1+\alpha}}\mathcal{D}_{n-1}.$$
Therefore, returning to \eqref{eq: delta_T^n bound}, we obtain
$$\mathcal{D}_{n}\lesssim \int_0^1 \frac{\eta_*}{\langle s_1\rangle^\alpha}\mathcal{D}_{n-1}ds_1\lesssim \eta_*\mathcal{D}_{n-1},$$
where the implicit constants are independent of $n\geq 1$. From the above inequality, we deduce that $\{\Xi^{(n)}(s,t,x,v)\}_{n=1}^\infty$ is contractive in $C_{x,v}$ for $0\leq s\leq t\leq 1$. Then, its limit $\Xi(s,t,x,v)$ satisfies the ODE \eqref{eq: forward characteristic ODE for VP-0} with \eqref{eq: derivative characteristic bound} on $[0,1]$, and upon letting $f(t,x,v)=f_0(\Xi(0,t,x,v))$, this distribution function obeys the dispersion bound \eqref{eq: dispersion estimates-density} on $[0,1]$.

Next, we extend the solution $f(t,x,v)$ globally in time via a standard continuity argument. We define
\[
\mu(t) = \sup_{0 \le s \le t}\Big\{\langle s\rangle^{\alpha+1}\|\mathbf{E}_f(s)\|_{L_x^\infty}
+\langle s\rangle^{\alpha+2}\|\nabla_x \mathbf{E}_f(s)\|_{L_x^\infty}\Big\},
\]
where $\mathbf{E}_f$ denotes the associated force field. Then, passing to the limit in $n$ in~\eqref{eq: force field density estimate}, we have $\mu(t)\le C\eta_*$ for some $t\ge1$. On the other hand, let
\[
T_{\mathrm{max}} = \sup\big\{t\ge0 : \mu(t)\le \eta_0\big\}
\]
for a sufficiently small $\eta_0 \gg \eta_* > 0$. 
In particular, letting $K > 0$ be the constant determined within Lemma \ref{lemma: dispersion bounds imply characteristic flow estimates}, we take $\eta_0 < \frac{1}{10K}$.
Then, under the bootstrap assumption $\mu(t)\le\eta_0$ for all $t \in [0,T_{\mathrm{max}} )$, and we repeat the previous argument using Lemmas~\ref{lemma: dispersion bounds imply characteristic flow estimates}, \ref{lemma: characteristic flow bounds imply dispersion estimates}, and \ref{lemma: interpolation inequality}, now applied to $\mathbf{E}_f$. This yields
$$\| \rho_{f(t)} \|_{W_x^{1,1}} +
\langle t \rangle^{3} \| \rho_{f(t)} \|_{W_x^{1,\infty}}
\le 4 \eta_*$$
and 
\begin{equation}\label{eq: force field/its derivative decay bounds}
\langle t\rangle^{\alpha+1}\|\mathbf{E}_f(t)\|_{L_x^\infty}
+\langle t\rangle^{\alpha+2}\|\nabla \mathbf{E}_f(t)\|_{L_x^\infty}
\le C\eta_*
\end{equation}
for all $t\in[0,T_{\mathrm{max}})$. Therefore, choosing $\eta_*$ smaller, if necessary, we improve the bootstrap assumption to $\mu(t)\le\frac{1}{2}\eta_0$ for all $t \in [0,T_{\mathrm{max}})$. Hence, we conclude $T_{\mathrm{max}}=\infty$, which implies global-in-time existence of the solution with the force field bounds \eqref{eq: force field/its derivative decay bounds} for all $t\geq0$.

Finally, it remains to show the bound for $\nabla^2 \mathbf{E}_f(t)$ in \eqref{eq: dispersion estimates}. Indeed, using a smooth cut-off $\chi \in C_c^\infty$ such that $\chi \equiv 1$ in $|x| \leq 1$ and $\chi$ is supported in $|x| \leq 2$, we have
$$
\nabla_{x_j}\nabla_{x_k} \mathbf{E}_f(t) =  \nabla_{x_j} (\nabla w \chi)* \nabla_{x_k} \rho_{f}(t) + \nabla_{x_j}\nabla_{x_k} \left( \nabla w (1 - \chi) \right) * \rho_{f}(t).
$$
Thus, by \eqref{eq: dispersion estimates-density}, we prove that
$$
\|\nabla^2 \mathbf{E}_f(t)\|_{L_x^\infty} \lesssim \|\rho_{f}(t)\|_{W_x^{1,\infty}} \lesssim \frac{\eta_*}{\langle t \rangle^3},
$$
because $\nabla(\nabla w \chi)$ and $\nabla^2(\nabla w(1 - \chi)) \in L^1(\mathbb{R}^3)$.
\end{proof}

\subsection{Bounds for the forward-in-time characteristic flows}\label{sec: Bounds for the forward-in-time characteristic flows}

Let $f(t)$ be the global solution constructed in Theorem~\ref{thm: global well-posedness}.  
For each $(x,v) \in \mathbb{R}^3 \times \mathbb{R}^3$, we define the forward-in-time characteristic flow \((\mathcal{X}(t,0,x,v), \mathcal{V}(t,0,x,v))\) as the solution to the system \eqref{eq: forward characteristic ODE for VP-0}, where $\mathbf{E}_f$ is the force field associated with the Vlasov--Riesz system \eqref{eq: VR force field}.  
Equivalently, in integral form, it can be written as
\begin{equation}\label{eq: characteristic flow, integral form}
\left\{
\begin{aligned}
\mathcal{X}(t,0,x,v) &= x + t v + \int_0^t (t-t_1) \, \mathbf{E}_f\big(t_1, \mathcal{X}(t_1,0,x,v)\big) \, dt_1,\\
\mathcal{V}(t,0,x,v) &= v + \int_0^t \mathbf{E}_f\big(t_1, \mathcal{X}(t_1,0,x,v)\big) \, dt_1.
\end{aligned}
\right.
\end{equation}

\begin{lemma}[Forward-in-time characteristic flow bounds]\label{lemma: forward-in-time characteristic flow bounds}
Let $0<\alpha<1$. Under the assumptions of Theorem \ref{thm: global well-posedness}, let 
\((\mathcal{X}(t,0,x,v), \mathcal{V}(t,0,x,v))\) denote the forward-in-time characteristic flow defined in \eqref{eq: forward characteristic ODE for VP-0}, associated with the force field $\mathbf{E}_f$ from Theorem~\ref{thm: global well-posedness}, and set 
\[
\mathcal{Y}(t,0,x,v) := \mathcal{X}(t,0,x,v) - t \, \mathcal{V}(t,0,x,v).
\]
Then, the following bounds hold:
\begin{align}
\sup_{t \ge 0}
\Bigg\{\frac{\| \nabla_x \mathcal{Y}(t,0,x,v)-\mathbb{I}_3\|_{C_{x,v}(\mathbb{R}^6)}}{t^{1-\alpha}}+\frac{\| \nabla_v \mathcal{Y}(t,0,x,v)\|_{C_{x,v}(\mathbb{R}^6)}}{t^{1-\alpha}}\Bigg\} &\lesssim \eta_*,\label{eq: forward-in-time characteristic flow bounds-1}\\
\sup_{t \ge 0} \Big(\| \nabla_x \mathcal{V}(t,0,x,v)\|_{C_{x,v}(\mathbb{R}^6)}+\| \nabla_v \mathcal{V}(t,0,x,v) -  \mathbb{I}_3\|_{C_{x,v}(\mathbb{R}^6)}\Big) &\lesssim \eta_*,\label{eq: forward-in-time characteristic flow bounds-2}\\
\sup_{t \ge 0}  
\Bigg\{
\frac{\|\nabla_v^2\mathcal{Y}(t,0,x,v)\|_{C_{x,v}(\mathbb{R}^6)}}{t}
+\frac{\|\nabla_v^2\mathcal{V}(t,0,x,v) \|_{C_{x,v}(\mathbb{R}^6)}}{\ln \langle t\rangle}\Bigg\} &\lesssim \eta_*.\label{eq: forward-in-time characteristic flow bounds-3}
\end{align}
\end{lemma}

\begin{proof}
Fix $(x,v) \in \mathbb{R}^3 \times \mathbb{R}^3$ and denote
\[
(\mathcal{X}_t, \mathcal{V}_t) := \big(\mathcal{X}(t,0,x,v), \mathcal{V}(t,0,x,v)\big)\quad\textup{and} 
\quad \mathcal{Y}_t := \mathcal{X}_t - t \mathcal{V}_t.
\]
Note from \eqref{eq: characteristic flow, integral form} that 
$$(\mathcal{Y}_t,\mathcal{V}_t) = \bigg(x - \int_0^t t_1 \, \mathbf{E}_f(t_1, \mathcal{X}_{t_1}) \, dt_1, v+\int_0^t \mathbf{E}_f(t_1, \mathcal{X}_{t_1}) \, dt_1\bigg).$$
Differentiating the above integral equations, we write
$$\left\{\begin{aligned}
\nabla_{(x,v)} \mathcal{Y}_t&=[\mathbb{I}_3 \ \ 0]^\top  - \int_0^t t_1 \, \nabla \mathbf{E}_f(t_1, \mathcal{X}_{t_1}) \, \nabla_{(x,v)} \mathcal{X}_{t_1} \, dt_1,\\
\nabla_{(x,v)} \mathcal{V}_t &= [0 \ \ \mathbb{I}_3]^\top+ \displaystyle\int_0^t \nabla \mathbf{E}_f(t_1, \mathcal{X}_{t_1}) \, \nabla_{(x,v)} \mathcal{X}_{t_1} \, dt_1.
\end{aligned}\right.$$
Then, by the dispersion estimates \eqref{eq: dispersion estimates}, we obtain
\[
\begin{aligned}
\mathbf{S}(t):&=\frac{1}{t^{1-\alpha}}\big\|\nabla_{(x,v)} \mathcal{Y}_t - [\mathbb{I}_3 \ \ 0]^\top\big\|+ \big\|\nabla_{(x,v)} \mathcal{V}_t - [0\ \ \mathbb{I}_3]^\top\big\|\\
&\lesssim\int_0^t \bigg(\frac{\eta_*}{t^{1-\alpha}\langle t_1 \rangle^{1+\alpha}}+\frac{\eta_*}{\langle t_1 \rangle^{2+\alpha}}\bigg) \|\nabla_{(x,v)} \mathcal{X}_{t_1}\| \, dt_1.
\end{aligned}
\]
Note that
$$\begin{aligned}
\|\nabla_{(x,v)} \mathcal{X}_{t_1}\|&\leq \|\nabla_{(x,v)} \mathcal{Y}_{t_1}\|+t_1\|\nabla_{(x,v)} \mathcal{V}_{t_1}\|\\
&\leq 1+t_1+\big\|\nabla_{(x,v)} \mathcal{Y}_{t_1}- [\mathbb{I}_3 \ \ 0]^\top\big\|+t_1\big\|\nabla_{(x,v)} \mathcal{V}_{t_1}- [0 \ \ \mathbb{I}_3]^\top\big\|\\
&\leq 1+t_1+t_1^{1-\alpha}\mathbf{S}(t_1)+t_1\mathbf{S}(t_1).
\end{aligned}$$
Thus, it follows that 
\[
\mathbf{S}(t)\lesssim \eta_* \left [ 1+ \int_0^t \bigg(\frac{1}{t^{1-\alpha} \langle t_1 \rangle^\alpha}+\frac{1}{\langle t_1 \rangle^{1+\alpha}}\bigg)\mathbf{S}(t_1) \, dt_1 \right ].
\]
With this, Grönwall's inequality yields $\mathbf{S}(t)\lesssim \eta_*$, and thus \eqref{eq: forward-in-time characteristic flow bounds-1} and \eqref{eq: forward-in-time characteristic flow bounds-2}.

Next, differentiating the equation for $(\mathcal{Y}_t,\mathcal{V}_t)$ twice with respect to $v$, we write 
\[\left\{
\begin{aligned}
\nabla_{v_j} \nabla_{v_k} \mathcal{Y}_t &= -\int_0^t t_1 \, \nabla_{v_j} \mathcal{X}_{t_1} \cdot (\nabla^2 \mathbf{E}_f)(t_1, \mathcal{X}_{t_1}) \, \nabla_{v_k} \mathcal{X}_{t_1}  \\
&\qquad\qquad+t_1 (\nabla \mathbf{E}_f)(t_1, \mathcal{X}_{t_1}) \, \nabla_{v_j} \nabla_{v_k} (\mathcal{Y}_{t_1} + t_1 \mathcal{V}_{t_1}) \, dt_1,\\
\nabla_{v_j} \nabla_{v_k} \mathcal{V}_t &= \int_0^t \nabla_{v_j}\mathcal{X}_{t_1} \cdot (\nabla^2 \mathbf{E}_f)(t_1, \mathcal{X}_{t_1}) \, \nabla_{v_k} \mathcal{X}_{t_1}\\
&\qquad\quad +  (\nabla \mathbf{E}_f)(t_1, \mathcal{X}_{t_1}) \, \nabla_{v_j} \nabla_{v_k} (\mathcal{Y}_{t_1} + t_1 \mathcal{V}_{t_1}) \, dt_1.
\end{aligned}\right.
\]
Note that by \eqref{eq: forward-in-time characteristic flow bounds-1} and \eqref{eq: forward-in-time characteristic flow bounds-2}, we have $\|\nabla_{v_j} \mathcal{X}_{t_1}\|\leq \|\nabla_{v_j} \mathcal{Y}_{t_1}\|+t_1\|\nabla_{v_j} \mathcal{V}_{t_1}\|\lesssim\langle t_1\rangle$. Hence, by the dispersion estimate \eqref{eq: dispersion estimates}, we obtain
\[\begin{aligned}
\frac{\|\nabla_v^2 \mathcal{Y}_t\|}{t} + \frac{\|\nabla_v^2 \mathcal{V}_t\|}{\ln\langle t\rangle} &\lesssim \frac{1}{t}\int_0^t\eta_*+t_1\frac{\eta_*}{\langle t_1\rangle^{2+\alpha}} \Big(\|\nabla_v^2 \mathcal{Y}_{t_1}\| +|t_1| \|\nabla_v^2 \mathcal{V}_{t_1}\| \Big)\, dt_1\\
&\quad+ \frac{1}{\ln\langle t\rangle}\int_0^t\frac{\eta_*}{\langle t_1\rangle}+\frac{\eta_*}{\langle t_1\rangle^{2+\alpha}} \Big(\|\nabla_v^2 \mathcal{Y}_{t_1}\| +|t_1| \|\nabla_v^2 \mathcal{V}_{t_1}\| \Big)\, dt_1\\
&\lesssim \eta_* + \frac{1}{t}\int_0^t \frac{\eta_*}{\langle t_1 \rangle^{\alpha}}\ln\langle t_1\rangle \bigg( \frac{\|\nabla_v^2 \mathcal{Y}_{t_1}\|}{t_1} + \frac{\|\nabla_v^2 \mathcal{V}_{t_1}\|}{\ln\langle t_1\rangle} \bigg)\, dt_1\\
&\quad +\frac{1}{\ln\langle t\rangle}\int_0^t\frac{\eta_*}{\langle t_1\rangle^{1+\alpha}}\ln \langle t_1\rangle \bigg( \frac{\|\nabla_v^2 \mathcal{Y}_{t_1}\|}{t_1} + \frac{\|\nabla_v^2 \mathcal{V}_{t_1}\|}{\ln\langle t_1\rangle} \bigg)\, dt_1.
\end{aligned}\]
Therefore, applying Grönwall's inequality, we arrive at \eqref{eq: forward-in-time characteristic flow bounds-3}.
\end{proof}

\section{Construction of the modified wave operator for characteristic flows}\label{sec: Construction of the modified wave operator}

For the global solution $f(t)$to the Vlasov--Riesz system \eqref{eq: VR}, constructed in Theorem \ref{thm: global well-posedness}, we consider the associated Hamiltonian flow $\Phi(t) = (\Phi_1(t), \Phi_2(t))$ defined by the characteristic ODE
\begin{equation}\label{eq:forward_characteristic_ODE_for_VP}
\left\{
\begin{aligned}
\partial_t \Phi(t)(x,v) &= \big(\Phi_2(t)(x,v), \mathbf{E}_f(t, \Phi_1(t)(x,v))\big),\\
\Phi(0)(x,v) &= (x,v),
\end{aligned}
\right.
\end{equation}
where $\mathbf{E}_f=\mathbf{E}_f(t,x)$ denotes the force field. 
In this section, we introduce a reference flow $\tilde{\Phi}^{\mathrm{ref}}(t)$ capturing the leading-order behavior of $\Phi(t)$ and construct the corresponding finite- and infinite-time modified wave operators.

\subsection{Construction of the modified reference flow}\label{sec: construction of the modified reference flow}

Recall the notation
\[
\Phi(t)(x,v) = \big(\mathcal{X}(t,0,x,v), \mathcal{V}(t,0,x,v)\big),
\]
for $(x,v) \in \mathbb{R}^3 \times \mathbb{R}^3$, introduced in Section~\ref{sec: Bounds for the forward-in-time characteristic flows}.
The dispersion estimate for the force field (see \eqref{eq: dispersion estimates}) implies that the momentum $\mathcal{V}(t,0,x,v)$ converges as $t \to \infty$.
We define the limiting momentum by
\begin{equation}\label{eq: momentum limit}
\mathcal{V}^+(x,v) = \mathcal{V}(\infty,0,x,v):= v + \int_0^\infty \mathbf{E}_f\big(t, \mathcal{X}(t,0,x,v)\big) \, dt,
\end{equation}
and establish the following quantitative convergence estimates.

\begin{lemma}[Momentum limit]\label{lemma: properties of the momentum limit}
Under the assumptions of Theorem \ref{theorem: modified scattering}, the momentum characteristics $\mathcal{V}(t,0,x,v)$ satisfy the following bounds:
\begin{align}
\sup_{t\in[0,\infty)}\left (\langle t\rangle^\alpha\|\mathcal{V}(t,0,x,v)-\mathcal{V}^+(x,v)\|_{L_{x,v}^\infty(\mathbb{R}^6)} \right )&\lesssim \eta_*,\label{eq: rate of convergence for momentum limit}\\
\sup_{t\in[0,\infty)\cup\{\infty\}}\|\mathcal{V}(t,0,x,v)-v\|_{L_{x,v}^\infty(\mathbb{R}^6)}&\lesssim \eta_*,\label{eq: almost identity of V(t,0,x,v)}\\
\sup_{t\in[0,\infty)\cup\{\infty\}}\big\|\nabla_{(x,v)} \mathcal{V}(t,0,x,v)-[0\ \mathbb{I}_3]^\top\big\|_{L_{x,v}^\infty(\mathbb{R}^6)}&\lesssim \eta_*.\label{eq: v derivative of V limit}
\end{align}
\end{lemma}

\begin{proof}
From the representation
\[
\mathcal{V}(t,0,x,v) = \mathcal{V}^+(x,v) - \int_t^\infty \mathbf{E}_f\big(t_1, \mathcal{X}(t_1,0,x,v)\big) \, dt_1
\]
of \eqref{eq: characteristic flow, integral form}, the dispersion estimate \eqref{eq: dispersion estimates} yields \eqref{eq: rate of convergence for momentum limit}. 
Similarly, using the integral form
\begin{equation}\label{eq: V(t,0,x,v) integral form}
\mathcal{V}(t,0,x,v) = v + \int_0^t \mathbf{E}_f\big(t_1, \mathcal{X}(t_1,0,x,v)\big) \, dt_1
\end{equation}
(valid also for $t=\infty$) together with \eqref{eq: dispersion estimates}, we obtain \eqref{eq: almost identity of V(t,0,x,v)}. 
Moreover, differentiating \eqref{eq: V(t,0,x,v) integral form} and applying \eqref{eq: dispersion estimates} along with the bound $\|\nabla_{(x,v)} \mathcal{X}(t,0,x,v)\| \lesssim \langle t \rangle$ from the proof of Lemma~\ref{lemma: forward-in-time characteristic flow bounds}, we obtain
\[
\begin{aligned}
\big\|\nabla_{(x,v)} \mathcal{V}(t,0,x,v)-[0\ \mathbb{I}_3]^\top\big\|
&\lesssim \int_0^\infty \|\nabla \mathbf{E}_f(t)\|_{L_x^\infty}
\, \|\nabla_{(x,v)} \mathcal{X}(t,0,x,v)\| \, dt\\
&\lesssim \int_0^\infty \frac{\eta_*}{\langle t \rangle^{2+\alpha}} \langle t \rangle \, dt
\lesssim \eta_*,
\end{aligned}
\]
which proves \eqref{eq: v derivative of V limit}.
\end{proof}

For the position $\mathcal{X}(t,0,x,v)$, subtracting the free flow yields the integral representation\footnote{This formula is obtained by integrating \eqref{eq:forward_characteristic_ODE_for_VP} from $t=1$. The lower limit $t=0$ is excluded to avoid a technical difficulty caused by the singularity at $t=0$.}
\[
\mathcal{X}(t,0,x,v)
= \mathcal{Y}(1,0,x,v)
+ t \, \mathcal{V}(t,0,x,v)
\underbrace{- \int_1^t t_1 \, \mathbf{E}_f\big(t_1, \mathcal{X}(t_1,0,x,v)\big)\, dt_1}_{(*)}
\]
for $t \geq 1$, where $\mathcal{Y}(t,0,x,v)=\mathcal{X}(t,0,x,v)-t\mathcal{V}(t,0,x,v)$. A key observation is that, due to the long-range interaction, the dispersion estimate \eqref{eq: dispersion estimates} does not ensure the convergence of $(*)$ as $t \to \infty$, unlike the momentum convergence \eqref{eq: momentum limit}. Hence, we must extract the contribution of the next order terms. Expanding the integral gives
\[
\begin{aligned}
(*)&= \int_1^t t_1 \iint_{\mathbb{R}^3 \times \mathbb{R}^3}
\nabla w\big(\mathcal{X}(t_1,0,x,v) - \tilde{x}\big)
f_0\Big(\mathcal{X}(0,t_1,\tilde{x},\tilde{v}), \mathcal{V}(0,t_1,\tilde{x},\tilde{v})\Big)
\, d\tilde{x} d\tilde{v}\, dt_1 \\
&= \int_1^t t_1 \iint_{\mathbb{R}^3 \times \mathbb{R}^3}
\nabla w\big(\mathcal{X}(t_1,0,x,v) - \mathcal{X}(t_1,0,\tilde{x},\tilde{v})\big)
f_0(\tilde{x},\tilde{v}) \, d\tilde{x} d\tilde{v}\, dt_1,
\end{aligned}
\]
where in the last equality, we have used the change of variables
\[
\big(\mathcal{X}(0,t_1,\tilde{x},\tilde{v}), \mathcal{V}(0,t_1,\tilde{x},\tilde{v})\big)
\mapsto (\tilde{x}, \tilde{v}).
\]
Note that by the dispersion estimate \eqref{eq: dispersion estimates} and Lemma~\ref{lemma: properties of the momentum limit},
\begin{equation}\label{eq: asymptotic with correction}
\begin{aligned}\big(\mathcal{X}(t_1,0,x,v) - \mathcal{X}(t_1,0,\tilde{x},\tilde{v})\big)&\approx (x - \tilde{x})
+t_1\big( \mathcal{V}(t_1,0,x,v) - \mathcal{V}(t_1,0,\tilde{x},\tilde{v})\big)
\\
&\approx t_1\big( \mathcal{V}(t,0,x,v) - \mathcal{V}(t,0,\tilde{x},\tilde{v})\big)
\end{aligned}
\end{equation}
for $t \ge t_1 \gg 1$. Replacing $\mathcal{X}(t_1,0,x,v) - \mathcal{X}(t_1,0,\tilde{x},\tilde{v})$ by
$t_1\big(\mathcal{V}(t,0,x,v) - \mathcal{V}(t,0,\tilde{x},\tilde{v})\big)$
for $t_1$ sufficiently large
and using the homogeneity of \(\nabla w(x) = -\alpha \frac{x}{|x|^{2+\alpha}}\), we extract the leading-order term
\[
\begin{aligned}
&\int_1^t t_1\iint_{\mathbb{R}^3 \times \mathbb{R}^3}
\nabla w\Big(t_1 \big(\mathcal{V}(t,0,x,v) - \mathcal{V}(t,0,\tilde{x},\tilde{v})\big)\Big)
f_0(\tilde{x},\tilde{v}) \, d\tilde{x} d\tilde{v}\, dt_1 \\
&= -\int_1^t \frac{1}{t_1^\alpha} \, \mathbf{A}_t\big(\mathcal{V}(t,0,x,v)\big)\, dt_1
= - \frac{t^{1-\alpha}-1}{1-\alpha} \,
\mathbf{A}_t\big(\mathcal{V}(t,0,x,v)\big),
\end{aligned}
\]
where
\begin{equation}\label{eq: A_t(v) definition}
\mathbf{A}_t(v)
:= - \iint_{\mathbb{R}^6}
\nabla w\big(v - \mathcal{V}(t,0,\tilde{x},\tilde{v})\big)
\, f_0(\tilde{x},\tilde{v}) \, d\tilde{x} d\tilde{v}.
\end{equation}
In particular, we have shown
\begin{equation}
\label{Econv}
\mathbf{E}_f\big(t_1, \mathcal{X}(t_1,0,x,v)\big) \approx \frac{1}{t_1^{1+\alpha}} \mathbf{A}_t(\mathcal{V}(t,0,x,v))
\end{equation}
for $t \geq t_1 \gg 1$.
Consequently, we expect the asymptotic expansion
\[
\begin{aligned}
\mathcal{X}(t,0,x,v)
&= \mathcal{Y}(1,0,x,v)
+ t \, \mathcal{V}(t,0,x,v)
- \frac{t^{1-\alpha}-1}{1-\alpha} \, \mathbf{A}_t\big(\mathcal{V}(t,0,x,v)\big) \\
&\quad + \big(\text{a remainder converging as } t \to \infty \big).
\end{aligned}
\]
By Lemma \ref{lemma: properties of the momentum limit}, we may also define the limit case $t=\infty$ in \eqref{eq: A_t(v) definition} as
\begin{equation}\label{eq: A(v) definition}
\mathbf{A}_\infty(v)
:= - \iint_{\mathbb{R}^6}
\nabla w\big(v - \mathcal{V}^+(\tilde{x},\tilde{v})\big)
\, f_0(\tilde{x},\tilde{v}) \, d\tilde{x} d\tilde{v}.
\end{equation}

The following lemma shows that $\mathbf{A}_t(v)$ (resp., $\mathbf{A}_\infty(v)$) is a well-defined $C^2$ (resp., $C^1$) function.

\begin{lemma}[Bounds for $\mathbf{A}_t$]\label{lemma: derivative bound for A(v)}
Under the assumptions in Theorem \ref{theorem: modified scattering}, we have
\begin{align}
\sup_{t\in[0,\infty)\cup\{\infty\}}\| \mathbf{A}_t(v)\|_{C^1(\mathbb{R}^3)}&\lesssim \eta_*,\label{eq: derivative bound for A(v)}\\
\sup_{t\in[0,\infty)} \left (\frac{1}{\ln\langle t\rangle}\|\nabla^2\mathbf{A}_t(v)\|_{C(\mathbb{R}^3)} \right) &\lesssim \eta_*.\label{eq: derivative bound for A_t(v)}
\end{align}
Moreover, we have
\begin{equation}\label{eq: estimate for t derivative of A_t(v)}
\sup_{t\in[0,\infty)} \left [\langle t\rangle^{1+\alpha}\bigg(\|\partial_t\mathbf{A}_t(v)\|_{C(\mathbb{R}^3)}+\frac{1}{\ln\langle t\rangle}\|\partial_t\nabla\mathbf{A}_t(v)\|_{C(\mathbb{R}^3)}\bigg) \right ]\lesssim \eta_*,
\end{equation}
and thus,
\begin{equation}\label{eq: convergence estimate for A_t(v)}
\sup_{t\geq0} \biggl (\langle t\rangle^\alpha\|\mathbf{A}_t(v)-\mathbf{A}_\infty(v)\|_{C(\mathbb{R}^3)} \biggr )\lesssim  \eta_*.
\end{equation}
\end{lemma}
\begin{remark}
Combining equation \eqref{eq: convergence estimate for A_t(v)} with equation \eqref{Econv}, and using the weighted (in $x$) norm bound from \eqref{eq: small data condition}, one may further show convergence of the field, namely
$$\sup_{v \in \mathbb{R}^3} \left | t^{1+\alpha} \mathbf{E}_f(t,tv) - \mathbf{A}_\infty(v) \right | \lesssim \langle t\rangle^{-\alpha} .$$
\end{remark}

\begin{proof}
To establish \eqref{eq: derivative bound for A(v)}, we let $m=(m_1,m_2,m_3)\in\mathbb{Z}_{\geq 0}^3$ be a multi-index with $|m|=0,1$. Then, using \eqref{eq: v derivative of V limit} to change variables via \(y=\mathcal{V}(t,0,\tilde{x},\tilde{v})\) in the definition \eqref{eq: A_t(v) definition}, we obtain
\begin{align*}
\big|(\nabla^m_v \mathbf{A}_t)(v)\big|
&\leq \iint_{\mathbb{R}^6}
\big|(\nabla^m \nabla w)(v-y)\big|
f_0\big(\tilde{x},\tilde{v}(y)\big)\big|\textup{det}\big(\nabla_{\tilde{v}}y\big)\big|^{-1}\, d\tilde{x}dy \\
&\sim \iint_{\mathbb{R}^6}
\big|(\nabla^m \nabla w)(v-y)\big|
f_0\big(\tilde{x},\tilde{v}(y)\big)
d\tilde{x}dy\\
&\leq \int_{\mathbb{R}^3}
\|\nabla^m \nabla w\|_{L^1+L^\infty}\big\|
f_0\big(\tilde{x},\tilde{v}(y)\big)\big\|_{L_y^\infty\cap L_y^1}
d\tilde{x}\lesssim \|f_0\|_{L_{\tilde{x}}^1L_{\tilde{v}}^\infty\cap L_{\tilde{x},\tilde{v}}^1}.
\end{align*}
To prove \eqref{eq: derivative bound for A_t(v)}, we again change variables by \(\tilde{v}\mapsto y=\mathcal{V}(t,0,\tilde{x},\tilde{v})\) so that
\begin{equation}\label{eq: 2nd derivative of A(t)}
\begin{aligned}
\nabla_{v_j}\nabla_{v_k}\mathbf{A}_t(v)
&= \nabla_{v_j}\nabla_{v_k}\iint_{\mathbb{R}^6}
\nabla w(v - y)
\, \Big\{f_0\big(\tilde{x},\tilde{v}(y)\big)
\big|\textup{det}\big(\nabla_{\tilde{v}}y\big)\big|^{-1}\Big\}\, d\tilde{x}dy\\
&= -\iint_{\mathbb{R}^6}
\nabla_{v_k}\nabla w(v - y)
\, \nabla_{y_j}\Big\{f_0\big(\tilde{x},\tilde{v}(y)\big)
\big|\textup{det}\big(\nabla_{\tilde{v}}y\big)\big|^{-1}\Big\}\, d\tilde{x}dy.
\end{aligned}
\end{equation}
Note that by \eqref{eq: forward-in-time characteristic flow bounds-2}, $(\nabla_{\tilde{v}}y)^{-1}$ can be written as an absolutely convergent series 
$$(\nabla_{\tilde{v}}y)^{-1}=\Big(\mathbb{I}_3+\big(\nabla_{\tilde{v}}y-\mathbb{I}_3\big)\Big)^{-1}=\mathbb{I}_3
+\sum_{k=1}^\infty (-1)^k\Big(\nabla_{\tilde{v}}\mathcal{V}(t,0,\tilde{x},\tilde{v})-\mathbb{I}_3\Big)^k\mathbb{I}_3$$
that is, as a small perturbation by sums of products of $(\nabla_{\tilde{v}_j}\mathcal{V}^\ell(t,0,\tilde{x},\tilde{v})-\delta_{j\ell})$ terms for each element of the identity matrix $\mathbb{I}_3$. Hence, using Lemma \ref{lemma: forward-in-time characteristic flow bounds}, we find
$$\Big|\nabla_{\tilde{v}}\Big(\textup{det}(\nabla_{\tilde{v}}y)^{-1}\Big)\Big|\lesssim \ln\langle t\rangle,$$
where the logarithmic bound is obtained when $\nabla_{\tilde{v}}$ is applied to $\nabla_{\tilde{v}}\mathcal{V}(t,0,\tilde{x},\tilde{v})$, due to the second derivative bound in \eqref{eq: forward-in-time characteristic flow bounds-3}. Thus, we have 
$$\begin{aligned}
\Big|\nabla_{y_j}\Big\{f_0\big(\tilde{x},\tilde{v}(y)\big)
\big|\textup{det}\big(\nabla_{\tilde{v}}y\big)\big|^{-1}\Big\}\Big|&\leq\big|(\nabla_vf_0)\big(\tilde{x},\tilde{v}(y)\big)\big||\nabla_{y_j}\tilde{v}(y)|
\big|\textup{det}\big(\nabla_{\tilde{v}}y\big)\big|^{-1}\\
&\quad+\big|f_0\big(\tilde{x},\tilde{v}(y)\big)\big|\Big|\nabla_{\tilde{v}}\Big(
\big|\textup{det}\big(\nabla_{\tilde{v}}y\big)\big|^{-1}\Big)\Big||\nabla_{y_j}\tilde{v}|\\
&\lesssim \big|(\nabla_vf_0)\big(\tilde{x},\tilde{v}(y)\big)\big|+\big|f_0\big(\tilde{x},\tilde{v}(y)\big)\big|\ln\langle t\rangle.
\end{aligned}$$
Then, applying this bound to \eqref{eq: 2nd derivative of A(t)} and repeating the estimates in the proof of \eqref{eq: derivative bound for A(v)}, the estimate \eqref{eq: derivative bound for A_t(v)} follows.

In order to show \eqref{eq: estimate for t derivative of A_t(v)} and \eqref{eq: convergence estimate for A_t(v)}, we first note that
$$\partial_t\mathbf{A}_t(v)
= - \sum_{\ell=1}^3\iint_{\mathbb{R}^6}
\nabla\nabla_{x_\ell} w\big(v - \mathcal{V}(t,0,\tilde{x},\tilde{v})\big)\cdot\mathbf{E}^\ell_f\big(t,\mathcal{X}(t,0,\tilde{x},\tilde{v})\big)
\, f_0(\tilde{x},\tilde{v}) \, d\tilde{x} d\tilde{v}$$
and
$$\begin{aligned}
\partial_t\nabla_{v_j}\mathbf{A}_t(v)
&= \sum_{\ell=1}^3\iint_{\mathbb{R}^6}
\nabla\nabla_{x_\ell} w(v -y)\\
&\qquad\qquad\qquad\cdot\partial_{y_j}\Big\{\mathbf{E}^\ell_f\big(t,\mathcal{X}(t,0,\tilde{x},\tilde{v}(y))\big)
f_0\big(\tilde{x},\tilde{v}(y)\big)\big|\textup{det}\big(\nabla_{\tilde{v}}y\big)\big|^{-1}\Big\} \, d\tilde{x} dy.
\end{aligned}$$
Then, applying the same argument with the dispersion bound \eqref{eq: dispersion estimates} for the additional $\mathbf{E}_f$ term, one obtains \eqref{eq: estimate for t derivative of A_t(v)}.  
The convergence estimate \eqref{eq: convergence estimate for A_t(v)} then follows from the fundamental theorem of calculus.
\end{proof}

In conclusion, by Lemma~\ref{lemma: properties of the momentum limit}, 
we introduce the reference flow $\tilde{\Phi}^{\textup{ref}}(t)$ defined by
\[
\tilde{\Phi}^{\textup{ref}}(t)(x,v)
:=\left(x+tv-\frac{t^{1-\alpha}-1}{1-\alpha}\mathbf{A}_t(v),\, v\right),
\]
which can be regarded as a refinement of the free flow
\[
\Phi^{\textup{free}}(t)(x,v)=(x+tv,v).
\]

\begin{remark}\label{remark: second derivative of A}
The reference flow $\tilde{\Phi}^{\textup{ref}}(t)$ differs slightly from 
\(\Phi^{\textup{ref}}(t) = \big(x + t v - \frac{t^{1-\alpha}-1}{1-\alpha} \mathbf{A}_\infty(v),\, v \big)\)
introduced in Theorem~\ref{theorem: modified scattering}.  
For technical reasons, it is more convenient to work with $\tilde{\Phi}^{\textup{ref}}(t)$, as the subsequent analysis requires control of the second derivative of $\mathbf{A}_t(v)$ (or $\mathbf{A}_\infty(v)$). 
Estimating $\nabla^2 \mathbf{A}_t(v)$ is considerably easier than estimating $\nabla^2 \mathbf{A}_\infty(v)$. Indeed, controlling $\nabla^2 \mathbf{A}_\infty(v)$ seems to require bounds on (at least fractional) derivatives of the density $\rho_{f(t)}$ with additional decay, which is challenging and, moreover, would necessitate higher regularity of the initial data. 
This difficulty can be avoided by using $\mathbf{A}_t(v)$ instead.
\end{remark}

\subsection{Finite-time modified wave operator}\label{sec: Finite-time modified wave operator}

We define the \emph{finite-time modified wave operator} by
\begin{equation}\label{eq: finite-time modified wave operator}
\mathcal{W}^{\textup{mod}}(t)
:= \tilde{\Phi}^{\textup{ref}}(t)^{-1}\circ \Phi(t)
:\,[0,\infty)\times\mathbb{R}^3\times\mathbb{R}^3
\to \mathbb{R}^3\times\mathbb{R}^3.
\end{equation}
More explicitly, it is given by
\[
\mathcal{W}^{\textup{mod}}(t)(x,v)
= \bigl(\mathcal{W}_1^{\textup{mod}}(t)(x,v),\, \mathcal{W}_2^{\textup{mod}}(t)(x,v)\bigr),
\]
where
\[
\left\{\begin{aligned}
\mathcal{W}_1^{\textup{mod}}(t)(x,v)
&=
\mathcal{Y}(t,0,x,v)+ \frac{t^{1-\alpha}-1}{1-\alpha}
\,\mathbf{A}_t\bigl(\mathcal{V}(t,0,x,v)\bigr),\\
\mathcal{W}_2^{\textup{mod}}(t)(x,v)
&=
\mathcal{V}(t,0,x,v),
\end{aligned}\right.
\]
because
$$\mathcal{Y}(t,0,x,v)=\mathcal{X}(t,0,x,v)
- t\,\mathcal{V}(t,0,x,v).$$

By the computations carried out in the previous subsection, the modified wave operator admits the following integral representation.
\begin{lemma}[Integral representation of the modified wave operator]\label{lemma: Integral representation of the modified wave operator}
$$\left\{\begin{aligned}
\mathcal{W}_1^{\textup{mod}}(t)(x,v)
&= \mathcal{Y}(1,0,x,v)
+ \int_1^t \frac{1}{t_1^\alpha}\iint_{\mathbb{R}^6}
\bigg[
\nabla w\bigg(
\frac{\mathcal{X}(t_1,0,x,v)
-\mathcal{X}(t_1,0,\tilde{x},\tilde{v})}{t_1}
\bigg) \\
&\qquad\qquad\qquad\qquad\quad
- \nabla w\bigl(
\mathcal{V}(t,0,x,v)
- \mathcal{V}(t,0,\tilde{x},\tilde{v})
\bigr)
\bigg]
f_0(\tilde{x},\tilde{v})
\, d\tilde{x}d\tilde{v}\, dt_1,\\
\mathcal{W}_2^{\textup{mod}}(t)(x,v)
&= \mathcal{V}(1,0,x,v) + \int_1^t \mathbf{E}_f\bigl(t_1, \mathcal{X}(t_1,0,x,v)\bigr)\, dt_1.
\end{aligned}\right.$$
\end{lemma}

\begin{proof}
By \eqref{eq: characteristic flow, integral form}, we obtain the formula for $\mathcal{W}_2^{\textup{mod}}(t)$ and 
\[
\mathcal{W}_1^{\textup{mod}}(t)(x,v)
=\mathcal{Y}(1,0,x,v) -\int_1^t t_1 \, \mathbf{E}_f\big(t_1, \mathcal{X}(t_1,0,x,v)\big)\, dt_1+ \frac{t^{1-\alpha}-1}{1-\alpha}
\,\mathbf{A}_t\bigl(\mathcal{V}(t,0,x,v)\bigr).
\]
Note that by changing variables and using the homogeneity of $\nabla w$, we arrive at
\[
\begin{aligned}
\mathbf{E}_f(t_1,x)&=-\iint_{\mathbb{R}^6} \, \nabla w\big(x-\tilde{x}\big)f_0\big(\mathcal{X}(0,t_1,\tilde{x},\tilde{v}), \mathcal{V}(0,t_1,\tilde{x},\tilde{v})\big)\, d\tilde{x}d\tilde{v}\\
&= -\frac{1}{t_1^{\alpha+1}}\iint_{\mathbb{R}^6}
\nabla w\bigg(
\frac{x
-\mathcal{X}(t_1,0,\tilde{x},\tilde{v})}{t_1}
\bigg) f_0(\tilde{x},\tilde{v})
\, d\tilde{x}d\tilde{v}.
\end{aligned}
\]
Thus, by the definition of $\mathbf{A}_t$ in \eqref{eq: A_t(v) definition}, we prove the representation for $\mathcal{W}_1^{\textup{mod}}(t)$.
\end{proof}

The following proposition shows that the finite-time modified wave operator
converges as $t \to \infty$, with an explicit rate.  
We define the limit
\begin{equation}\label{eq: infinite-time modified wave operator}
\mathcal{W}^{\textup{mod},+}(x,v)
:= \lim_{t\to\infty} \mathcal{W}^{\textup{mod}}(t)(x,v),
\end{equation}
and refer to it as the \emph{modified wave operator}.  
For notational convenience, when treating the finite-time and infinite-time
operators simultaneously, we set
\[
\mathcal{W}^{\textup{mod}}(\infty)
:= \mathcal{W}^{\textup{mod},+}.
\]

\begin{proposition}[Modified wave operator]
\label{prop: modified wave operator}
Under the assumptions of Theorem~\ref{theorem: modified scattering}, the following statements hold.
\begin{enumerate}[$(1)$]
\item \emph{(Existence and convergence rate)}
For each $(x,v)\in\mathbb{R}^6$, the limit
$\mathcal{W}^{\textup{mod},+}(x,v)$ exists. Moreover, for all $t\geq 1$, we have
\[
\left\{\begin{aligned}
\bigg\|
\frac{1}{\langle x\rangle}\Big(\mathcal{W}_1^{\textup{mod}}(t)(x,v)
- \mathcal{W}_1^{\textup{mod},+}(x,v)\Big)
\bigg\|_{L_{x,v}^\infty(\mathbb{R}^6)}
&\lesssim \frac{\eta_*}{t^{2\alpha-1}}, \\
\big\|
\mathcal{W}_2^{\textup{mod}}(t)(x,v)
- \mathcal{W}_2^{\textup{mod},+}(x,v)
\big\|_{L_{x,v}^\infty(\mathbb{R}^6)}
&\lesssim \frac{\eta_*}{t^{\alpha}} .
\end{aligned}\right.
\]

\item \emph{(Almost identity)}
For all $t\geq 1$ (including $t=\infty$), we have
\[
\bigg\|
\frac{1}{\langle x\rangle}\Big(\mathcal{W}_1^{\textup{mod}}(t)(x,v)-x\Big)
\bigg\|_{L_{x,v}^\infty(\mathbb{R}^6)}
+
\big\|
\mathcal{W}_2^{\textup{mod}}(t)(x,v)-v
\big\|_{L_{x,v}^\infty(\mathbb{R}^6)}
\lesssim \eta_* .
\]
\end{enumerate}
\end{proposition}

\begin{proof}
To prove $(1)$ it suffices to consider $\mathcal{W}_1^{\textup{mod}}(t)$, because by the definition $\mathcal{W}_2^{\textup{mod}}(t)(x,v)=\mathcal{V}(t,0,x,v)$, the estimate for $\mathcal{W}_2^{\textup{mod}}(t)(x,v)$ in~$(1)$ follows from \eqref{eq: rate of convergence for momentum limit}.

To show that the limit of $\mathcal{W}_1^{\textup{mod}}(t)(x,v)$ exists as $t \to \infty$, we consider the difference
\[
\mathcal{W}_1^{\textup{mod}}(t')(x,v) - \mathcal{W}_1^{\textup{mod}}(t)(x,v) \quad \big(t' \geq t \gg 1\big).
\]
For its integral representation, for each $0\leq\theta\leq1$, we introduce the interpolated variables 
\[\begin{aligned}
y^\theta :&= \theta \bigg\{ \frac{\mathcal{X}(t_1,0,x,v) - \mathcal{X}(t_1,0,\tilde{x},\tilde{v})}{t_1} \bigg\} + (1-\theta) \big\{ \mathcal{V}(t',0,x,v) - \mathcal{V}(t',0,\tilde{x},\tilde{v}) \big\},\\
z^\theta :&= \theta \big\{ \mathcal{V}(t',0,x,v)
- \mathcal{V}(t',0,\tilde{x},\tilde{v}) \big\}+ (1-\theta) \big\{\mathcal{V}(t,0,x,v) - \mathcal{V}(t,0,\tilde{x},\tilde{v}) \big\}.
\end{aligned}\]
Then, by Lemma \ref{lemma: Integral representation of the modified wave operator} with $y^0=z^1$, we write 
$$\begin{aligned}
\mathcal{W}_1^{\textup{mod}}(t')(x,v) - \mathcal{W}_1^{\textup{mod}}(t)(x,v)&= \int_1^{t'} \frac{1}{t_1^\alpha}\iint_{\mathbb{R}^6}
\big(\nabla w(y^1)- \nabla w(z^1)\big)
f_0(\tilde{x},\tilde{v})
\, d\tilde{x}d\tilde{v}\, dt_1\\
&\quad-\int_0^{t} \frac{1}{t_1^\alpha}\iint_{\mathbb{R}^6}
\big(\nabla w(y^1)- \nabla w(z^0)\big)
f_0(\tilde{x},\tilde{v})
\, d\tilde{x}d\tilde{v}\, dt_1\\
&= \int_t^{t'} \frac{1}{t_1^\alpha}\iint_{\mathbb{R}^6}
\big(\nabla w(y^1)- \nabla w(y^0)\big)
f_0(\tilde{x},\tilde{v})
\, d\tilde{x}d\tilde{v}\, dt_1\\
&\quad-\int_1^{t} \frac{1}{t_1^\alpha}\iint_{\mathbb{R}^6}
\big(\nabla w(z^1)- \nabla w(z^0)\big)
f_0(\tilde{x},\tilde{v})
\, d\tilde{x}d\tilde{v}\, dt_1.
\end{aligned}$$
Hence, replacing $\nabla w(z^1)$ by $\nabla w(y^0)$ in the first integral and applying the fundamental theorem of calculus, we obtain
\[
\begin{aligned}
&\mathcal{W}_1^{\textup{mod}}(t')(x,v) - \mathcal{W}_1^{\textup{mod}}(t)(x,v)\\
&=  \iint_{\mathbb{R}^6} \bigg\{\int_0^1 \int_{t}^{t'}\frac{1}{t_1^\alpha} \nabla^2 w\big(y^\theta\big) \frac{d y^\theta}{d \theta} dt_1 d\theta-\int_0^1 \int_0^t\frac{1}{t_1^\alpha} \nabla^2 w\big(z^\theta\big) \frac{d z^\theta}{d \theta}dt_1 d\theta\bigg\}f_0(\tilde{x},\tilde{v}) \, d\tilde{x} d\tilde{v}.
\end{aligned}
\]
As $\mathcal{Y}(t,0,x,v)=\mathcal{X}(t,0,x,v)-t\mathcal{V}(t,0,x,v)$, by \eqref{eq: characteristic flow, integral form} and \eqref{eq: dispersion estimates}, we have
\[
\begin{aligned}
\bigg|\frac{dy^\theta}{d\theta} - \frac{x-\tilde{x}}{t_1}\bigg|
&\leq \frac{|\mathcal{Y}(t_1,0,x,v) - x|}{t_1}
+\frac{|\mathcal{Y}(t_1,0,\tilde{x},\tilde{v}) - \tilde{x}|}{t_1} \\
&\quad + \big| \mathcal{V}(t',0,x,v) - \mathcal{V}(t_1,0,x,v) \big| + \big| \mathcal{V}(t',0,\tilde{x},\tilde{v}) - \mathcal{V}(t_1,0,\tilde{x},\tilde{v}) \big| \\
&\lesssim \frac{1}{t_1} \int_0^{t_1} t_2 \|\mathbf{E}_f(t_2)\|_{L_x^\infty} \, dt_2
+ \int_{t_1}^{t'} \|\mathbf{E}_f(t_2)\|_{L_x^\infty} \, dt_2
\lesssim t_1^{-\alpha}
\end{aligned}
\]
and
$$\begin{aligned}
\bigg|\frac{dz^\theta}{d\theta}\bigg|&\leq \big| \mathcal{V}(t',0,x,v) - \mathcal{V}(t,0,x,v) \big|+\big| \mathcal{V}(t',0,\tilde{x},\tilde{v}) - \mathcal{V}(t,0,\tilde{x},\tilde{v}) \big|\\
&\lesssim\int_t^{t'} \|\mathbf{E}_f(t_2)\|_{L_x^\infty} \, dt_2
\lesssim t^{-\alpha}
\end{aligned}$$
Hence, it follows that
\[
\begin{aligned}
\frac{|\mathcal{W}_1^{\textup{mod}}(t')(x,v) - \mathcal{W}_1^{\textup{mod}}(t)(x,v)|}{\langle x\rangle}&\lesssim \int_0^1 \int_{t}^{t'}\frac{1}{t_1^{2\alpha}} \iint_{\mathbb{R}^6}  \big| \nabla^2 w(y^\theta) \big| 
(\langle x\rangle f_0)(\tilde{x},\tilde{v}) \, d\tilde{x} d\tilde{v} \, dt_1 \, d\theta\\
&\quad + \frac{1}{t^\alpha} \int_0^1 \int_0^t \frac{1}{t_1^\alpha}\iint_{\mathbb{R}^6}  \big| \nabla^2 w(z^\theta) \big| f_0(\tilde{x},\tilde{v}) \, d\tilde{x} d\tilde{v} \, dt_1 \, d\theta.
\end{aligned}
\]
Within the integrals on the right side, we perform the change of variables $\tilde{v} \mapsto y^\theta=y^\theta(t,x,v,\tilde{x},\tilde{v})$ and $\tilde{v} \mapsto z^\theta=z^\theta(t,x,v,\tilde{x},\tilde{v})$, respectively. Indeed, we find from \eqref{eq: characteristic flow, integral form} that
\[
\begin{aligned}
\nabla_{\tilde{v}}y^\theta&= -\frac{\theta}{t_1} \nabla_{\tilde{v}} \mathcal{X}(t_1,0,\tilde{x},\tilde{v}) 
   - (1-\theta) \nabla_{\tilde{v}} \mathcal{V}(t',0,\tilde{x},\tilde{v}) \\
&=- \nabla_{\tilde{v}} \mathcal{V}(t',0,\tilde{x},\tilde{v})-\frac{\theta}{t_1}\nabla_{\tilde{v}}\mathcal{Y}(t_1,0,\tilde{x},\tilde{v})+\theta\nabla_{\tilde{v}}\Big(\mathcal{V}(t',0,\tilde{x},\tilde{v})-\mathcal{V}(t_1,0,\tilde{x},\tilde{v})\Big)\\
&= - \nabla_{\tilde{v}} \mathcal{V}(t',0,\tilde{x},\tilde{v})
+ \frac{\theta}{t_1} \int_0^{t_1} \tilde{t} \nabla \mathbf{E}_f\big(\tilde{t}, \mathcal{X}(\tilde{t},0,\tilde{x},\tilde{v})\big)\nabla_{\tilde{v}} \mathcal{X}(\tilde{t},0,\tilde{x},\tilde{v}) \, d\tilde{t}\\
&\quad   + \theta \int_{t_1}^{t'} \nabla \mathbf{E}_f\big(\tilde{t}, \mathcal{X}(\tilde{t},0,\tilde{x},\tilde{v})\big)\nabla_{\tilde{v}} \mathcal{X}(\tilde{t},0,\tilde{x},\tilde{v}) \, d\tilde{t}.
\end{aligned}
\]
Hence, using Lemma \ref{lemma: forward-in-time characteristic flow bounds}, we obtain
\begin{equation}\label{eq: y theta change of variable}
\big\| \nabla_{\tilde{v}} y^\theta + \mathbb{I}_3 \big\|
\le \big\| \nabla_{\tilde{v}} \mathcal{V}(t',0,\tilde{x},\tilde{v}) - \mathbb{I}_3 \big\|+ \frac{\theta}{t_1} \int_0^{t_1} \tilde{t} \frac{\eta_*}{\langle \tilde{t} \rangle^{2+\alpha}} \langle \tilde{t} \rangle \, d\tilde{t}
+ \theta \int_{t_1}^\infty \frac{\eta_*}{\langle \tilde{t} \rangle^{2+\alpha}} \langle \tilde{t} \rangle \, d\tilde{t}  \ll 1.
\end{equation}
Moreover, by Lemma \ref{lemma: forward-in-time characteristic flow bounds} again, we have
\begin{equation}\label{eq: z theta change of variable}
\big\|\nabla_{\tilde{v}}z^\theta+\mathbb{I}_3\big\| \leq \theta \big\| \nabla_{\tilde{v}} \mathcal{V}(t',0,\tilde{x},\tilde{v})-\mathbb{I}_3 \big\| + (1-\theta) \big\| \nabla_{\tilde{v}} \mathcal{V}(t,0,\tilde{x},\tilde{v})-\mathbb{I}_3 \big\|\ll1.
\end{equation}
Thus, changing the variables in the two integrals accordingly, it follows that 
\[
\begin{aligned}
&\frac{1}{\langle x\rangle}\big|\mathcal{W}_1^{\textup{mod}}(t')(x,v) - \mathcal{W}_1^{\textup{mod}}(t)(x,v)\big|\\
&\lesssim \int_0^1 \int_{t}^{t'}  \frac{1}{t_1^{2\alpha}} 
\iint_{\mathbb{R}^6} |\nabla^2 w(y^\theta)| (\langle x \rangle f_0)(\tilde{x}, \tilde{v}(y^\theta)) \, dy^\theta d\tilde{x} \, dt_1 \, d\theta\\
&\quad + \frac{1}{t^\alpha}\int_0^1 \int_1^t \frac{1}{t_1^\alpha}\iint_{\mathbb{R}^6}  \big| \nabla^2 w(z^\theta) \big|  f_0(\tilde{x},\tilde{v}(z^\theta)) \, dz^\theta d\tilde{x}  \, dt_1 \, d\theta.
\end{aligned}
\]
Then, decomposing 
\[
\nabla^2 w(x) = \nabla^2 w(x) \mathbbm{1}_{|x|\le 1} + \nabla^2 w(x) \mathbbm{1}_{|x|>1}
\]
in the integral and using H\"older's inequality with $\|\nabla^2 w \mathbbm{1}_{|\cdot|\le 1}\|_{L^1}, \|\nabla^2 w \mathbbm{1}_{|\cdot|>1}\|_{L^\infty}\lesssim1$, one can show that for $t' \geq t \gg 1$, 
\[\begin{aligned}
&\frac{1}{\langle x\rangle}\big|\mathcal{W}_1^{\textup{mod}}(t')(x,v) - \mathcal{W}_1^{\textup{mod}}(t)(x,v)\big|\\
&\lesssim \frac{1}{t^{2\alpha-1}}\Big(\big\|(\langle x \rangle f_0)\big(\tilde{x}, \tilde{v}(y^\theta)\big)\big\|_{L_{\tilde{x}}^1 L_{y^\theta}^\infty \cap L_{x,y^\theta}^1}+\big\|(\langle x \rangle f_0)\big(\tilde{x}, \tilde{v}(z^\theta)\big)\big\|_{L_{\tilde{x}}^1 L_{z^\theta}^\infty \cap L_{x,z^\theta}^1}\Big)\\
&\sim \frac{1}{t^{2\alpha-1}}\big\|\langle x \rangle f_0\big\|_{L_x^1 L_v^\infty \cap L_{x,v}^1}.
\end{aligned}\]
Therefore, for each $(x,v)$, the limit $\mathcal{W}_1^{\textup{mod},+}(x,v)$ exists. Taking $t'\to\infty$, we also obtain
\[
\frac{1}{\langle x\rangle}\big|\mathcal{W}_1^{\textup{mod}}(t')(x,v) - \mathcal{W}_1^{\textup{mod}}(t)(x,v)\big|
\lesssim \frac{1}{t^{2\alpha-1}} 
\|\langle x \rangle f_0\|_{L_x^1 L_v^\infty \cap L_{x,v}^1},
\]
which completes the proof of statement $(1)$ of the proposition.

Similarly, to show $(2)$, we recall the integral representation in Lemma \ref{lemma: Integral representation of the modified wave operator} and find
\[
\begin{aligned}
\mathcal{W}_1^{\textup{mod}}(t)(x,v) - x 
&= \int_0^t \iint_{\mathbb{R}^6} \frac{1}{t_1^\alpha} \bigg\{ 
\nabla w\bigg( \frac{\mathcal{X}(t_1,0,x,v) - \mathcal{X}(t_1,0,\tilde{x},\tilde{v})}{t_1} \bigg) \\
&\qquad\qquad\qquad\qquad - \nabla w\big( \mathcal{V}(t,0,x,v) - \mathcal{V}(t,0,\tilde{x},\tilde{v}) \big) 
\bigg\} f_0(\tilde{x},\tilde{v}) \, d\tilde{x} d\tilde{v}\,  dt_1
\end{aligned}
\]
and
\[
\mathcal{W}_2^{\textup{mod}}(t)(x,v) - v = \int_0^t \mathbf{E}_f\big(t_1, \mathcal{X}(t_1,0,x,v)\big) \, dt_1.
\]
Then, repeating the estimates in the proof of statement (1), statement (2) of the proposition follows.
\end{proof}

\section{Weighted $W^{1,\infty}$-estimates for modified distributions}
\label{sec: Energy estimates for modified distributions} 

Suppose that \( f(t) = f(t,x,v) \) is a global solution to the Vlasov--Riesz equation
\begin{equation}\label{eq: VR eq-energy estimates}
\partial_t f + v \cdot \nabla_x f + \mathbf{E}_f \cdot \nabla_v f = 0.
\end{equation}
We define the modified distribution function
\begin{equation}\label{eq: g(t) definition}
g(t,x,v) := f\left( t, x + t v - \frac{t^{1-\alpha}-1}{1-\alpha} \mathbf{A}_t(v), v \right)
= f\big(t, \tilde{\Phi}^{\mathrm{ref}}(t)(x,v)\big).
\end{equation}
The goal of this section is to establish global-in-time bounds for \( g(t) \).

\begin{proposition}[Global-in-time bounds for the modified distribution]\label{prop: global bounds for modified distribution}
Under the assumptions of Theorem~\ref{theorem: modified scattering}, the modified distribution \( g(t) \) satisfies
\begin{equation}\label{eq: g weight derivative L^2 bound}
\sup_{t \ge 1}\bigg(
\big\| \langle x \rangle g(t) \big\|_{L_{x,v}^\infty(\mathbb{R}^6)}
+ \big\| \langle x \rangle \nabla_x g(t) \big\|_{L_{x,v}^\infty(\mathbb{R}^6)}
+ \frac{1}{t^{1-\alpha}}
\big\| \langle x \rangle \nabla_v g(t) \big\|_{L_{x,v}^\infty(\mathbb{R}^6)}
\bigg)
\lesssim \eta_*.
\end{equation}
\end{proposition}

\begin{remark}
In Proposition~\ref{prop: global bounds for modified distribution}, the $O(t^{1-\alpha})$ bound is obtained for 
\(\|\langle x \rangle \nabla_v g(t)\|_{L_{x,v}^\infty}\), which follows from the weak decay estimate for \(\nabla_v \mathbf{F}_1\) in 
\eqref{eq: g eq. vector field bounds} (see Remark~\ref{remark: Equation for the modified distribution}).
We do not claim that this bound is optimal; however, it nos not necessary to improve this bound for the main result 
(see Remark~\ref{remark: modified scattering proof key remark}).
\end{remark}

For the proof, we first derive the equation for \(g(t)\) from \eqref{eq: VR eq-energy estimates}, 
together with suitable bounds.
\begin{lemma}[Equation for the modified distribution \( g \)]\label{lemma: Equation for the modified distribution}
Under the assumptions of Theorem~\ref{theorem: modified scattering}, the following results hold:
\begin{enumerate}[$(1)$]
\item The modified distribution \( g(t) \) satisfies the equation
\begin{equation}\label{eq: modified VR eq-energy estimates}
\partial_t g + \mathbf{F} \cdot \nabla_{(x,v)} g = 0,
\end{equation}
where the vector field \( \mathbf{F} = ( \mathbf{F}_1(t,x,v), \mathbf{F}_2(t,x,v) ) \), with the notation \( \mathbf{v} = (\mathbf{v}^1,\mathbf{v}^2,\mathbf{v}^3) \), is given by
\begin{equation}\label{eq: modified VR eq-energy estimates-2}
\left\{
\begin{aligned}
\mathbf{F}_1^j(t,x,v) &= -\left\{ t \mathbf{E}_f^j\big( t, \tilde{\Phi}_1^{\textup{ref}}(t)(x,v)\big) - \frac{1}{t^\alpha} \mathbf{A}_t^j(v)\right\}\\
&\quad+ \frac{t^{1-\alpha}-1}{1-\alpha}\Big\{\partial_t\mathbf{A}_t^j(v)+\sum_{\ell=1}^3\mathbf{E}_f^\ell\big( t, \tilde{\Phi}_1^{\textup{ref}}(t)(x,v)\big)\nabla_{v_\ell} \mathbf{A}_t^j(v)\Big\}, \\
\mathbf{F}_2^j(t,x,v) &= \mathbf{E}_f^j\big( t, \tilde{\Phi}_1^{\textup{ref}}(t)(x,v)\big).
\end{aligned}
\right.
\end{equation}
\item The vector field \( \mathbf{F} \) is divergence-free, i.e., 
\begin{equation}\label{eq: g eq. divergence-free vector field}
\nabla_{(x,v)} \cdot \mathbf{F} = \nabla_x \cdot \mathbf{F}_1 + \nabla_v \cdot \mathbf{F}_2 = 0.
\end{equation}

\item \( \mathbf{F}_1(t) \) and \( \mathbf{F}_2(t) \) satisfy the following bounds:
\begin{equation}\label{eq: g eq. vector field bounds}
\left\{
\begin{aligned}
\sup_{t \geq 1}\Big(t^{1+\alpha} \|\nabla_x \mathbf{F}_1(t)\|_{L_x^\infty(\mathbb{R}^3)} + t^{2+\alpha} \|\nabla_x \mathbf{F}_2(t)\|_{L_x^\infty(\mathbb{R}^3)}\Big) &\lesssim \eta_*, \\
\sup_{t \geq 1}\Big(t^\alpha \|\nabla_v \mathbf{F}_1(t)\|_{L_x^\infty(\mathbb{R}^3)} + t^{1+\alpha} \|\nabla_v \mathbf{F}_2(t)\|_{L_x^\infty(\mathbb{R}^3)}\Big) &\lesssim \eta_*
\end{aligned}
\right.
\end{equation}
and
\begin{equation}\label{eq: g eq. vector field bounds-2}
\sup_{t \geq 1} \left ( t^{2\alpha} \bigg\|\frac{1}{\langle x \rangle}\mathbf{F}_1(t)\bigg\|_{L_{x,v}^\infty(\mathbb{R}^6)} \right )\lesssim \eta_*.
\end{equation}
\end{enumerate}
\end{lemma}

\begin{remark}\label{remark: Equation for the modified distribution}
Within \eqref{eq: g eq. vector field bounds}, the worst decay rate, namely \( O(t^{-\alpha}) \), arises from \( \nabla_v \mathbf{F}_1(t) \). 
\end{remark}

\begin{proof}[Proof of Lemma \ref{lemma: Equation for the modified distribution}]

For convenience, and with a slight abuse of notation, we write
\begin{equation}\label{eq: E_f, x variable omitted}
\mathbf{E}_f(t, \cdots)=\mathbf{E}_f\Big(t, \tilde{\Phi}_1^{\textup{ref}}(t)(x,v)\Big)
\end{equation}
for the force field. This omission is harmless, since we ultimately take the \(L_x^\infty\)-norm of the force field and its derivatives.

For $(1)$, we use the definition of \( g(t) \) and the equation \eqref{eq: VR eq-energy estimates} for \(f(t)\) to observe that \( g \) solves the equation
\[
\begin{aligned}
\partial_t g(t,x,v) &= \left\{\partial_t f + v \cdot \nabla_x f - \left(\frac{1}{t^\alpha} \mathbf{A}_t(v) + \frac{t^{1-\alpha}-1}{1-\alpha}\partial_t \mathbf{A}_t(v)\right) \cdot \nabla_x f \right\} \big(t, \tilde{\Phi}^{\textup{ref}}(t)(x,v)\big) \\
&= - \left(\frac{1}{t^\alpha} \mathbf{A}_t(v) + \frac{t^{1-\alpha}-1}{1-\alpha}\partial_t \mathbf{A}_t(v)\right) \cdot \nabla_x g(t,x,v) -  \left(\mathbf{E}_f \cdot \nabla_v f\right)\big(t, \tilde{\Phi}^{\textup{ref}}(t)(x,v)\big).
\end{aligned}
\]
Indeed, by applying the Leibniz rule, we have
\[
(\nabla_{v_\ell} f)\big(t, \tilde{\Phi}^{\textup{ref}}(t)(x,v)\big) = \nabla_{v_\ell} g(t,x,v) - t \nabla_{x_\ell} g(t,x,v) +  \frac{t^{1-\alpha}-1}{1-\alpha} \nabla_{v_\ell} \mathbf{A}_t(v) \cdot \nabla_x g(t,x,v).
\]
Hence, using the notation \eqref{eq: E_f, x variable omitted}, we deduce 
\[
\begin{aligned}
\partial_t g(t,x,v) &= \left\{ t \mathbf{E}_f(t, \cdots) - \frac{1}{t^\alpha} \mathbf{A}_t(v)\right\} \cdot \nabla_x g(t,x,v) \\
&\quad - \frac{t^{1-\alpha}-1}{1-\alpha} \left\{\partial_t \mathbf{A}_t(v) + \sum_{\ell=1}^3 \mathbf{E}_f^\ell(t, \cdots) \nabla_{v_\ell} \mathbf{A}_t(v)\right\} \cdot \nabla_x g(t,x,v) \\
&\quad - \mathbf{E}_f(t, \cdots) \cdot \nabla_v g(t,x,v),
\end{aligned}
\]
which is precisely \eqref{eq: modified VR eq-energy estimates} with \eqref{eq: modified VR eq-energy estimates-2}. On the other hand, statement $(2)$ of the lemma follows immediately from the definitions in \eqref{eq: modified VR eq-energy estimates-2}.

It remains to show $(3)$. In fact, by \eqref{eq: g eq. divergence-free vector field}, it suffices to prove the bounds for \( \nabla_x \mathbf{F}_1(t) \), \( \nabla_v \mathbf{F}_1(t) \), and \( \nabla_x \mathbf{F}_2(t) \). Note that by direct calculations,
\begin{equation}\label{eq: derivatives of F vector field}
\left\{
\begin{aligned}
\nabla_{x_k} \mathbf{F}_1^j &= -t \nabla_{x_k} \mathbf{E}_f^j(t, \cdots) + \frac{t^{1-\alpha}-1}{1-\alpha} \sum_{\ell=1}^3 \nabla_{x_k} \mathbf{E}_f^\ell(t, \cdots) \nabla_{v_\ell} \mathbf{A}_t^j(v), \\
\nabla_{v_k} \mathbf{F}_1^j &= -t\nabla \mathbf{E}_f^j(t, \cdots) \cdot \left(t e_k - \frac{t^{1-\alpha}-1}{1-\alpha} \nabla_{v_k} \mathbf{A}_t(v)\right) + \frac{1}{t^\alpha} \nabla_{v_k} \mathbf{A}_t^j(v) \\
&\quad + \frac{t^{1-\alpha}-1}{1-\alpha} \sum_{\ell=1}^3\nabla \mathbf{E}_f^\ell(t, \cdots) \cdot \left(t e_k - \frac{t^{1-\alpha}-1}{1-\alpha} \nabla_{v_k} \mathbf{A}_t(v)\right) \nabla_{v_\ell} \mathbf{A}_t(v) \\
&\quad + \frac{t^{1-\alpha}-1}{1-\alpha} \left\{ \nabla_{v_k} \partial_t \mathbf{A}_t^j(v) + \sum_{\ell=1}^3 \mathbf{E}_f^\ell(t, \cdots) \cdot \nabla_{v_k} \nabla_{v_\ell} \mathbf{A}_t^j(v) \right\}, \\
\nabla_{x_k} \mathbf{F}_2^j &= \nabla_{x_k} \mathbf{E}_f^j(t, \cdots).
\end{aligned}
\right.
\end{equation}
From the dispersion bounds \eqref{eq: dispersion estimates} and Lemma~\ref{lemma: derivative bound for A(v)}, we obtain the following bounds:
\[
\left\{
\begin{aligned}
\|\nabla_x \mathbf{F}_1(t)\|_{L_x^\infty} &\lesssim \frac{\eta_*}{t^{2 + \alpha}} (t + t^{1-\alpha}) \lesssim \frac{\eta_*}{t^{1 + \alpha}}, \\
\|\nabla_v \mathbf{F}_1(t)\|_{L_x^\infty} &\lesssim t \frac{\eta_*}{t^{2 + \alpha}} (t + t^{1-\alpha}) + \frac{\eta_*}{t^\alpha} + t^{1 - \alpha} \frac{\eta_*}{t^{2 + \alpha}} (t + t^{1-\alpha}) + t^{1 - \alpha} \frac{\eta_* \ln \langle t \rangle}{t^{1 + \alpha}} \lesssim \frac{\eta_*}{t^\alpha}, \\
\|\nabla_x \mathbf{F}_2(t)\|_{L_x^\infty} &\lesssim \frac{\eta_*}{t^{2 + \alpha}}
\end{aligned}
\right.
\]
for \( t \geq 1 \), which establishes \eqref{eq: g eq. vector field bounds}.

It remains to show \eqref{eq: g eq. vector field bounds-2}. We note that
\[
\begin{aligned}
t \mathbf{E}_f(t, \cdots) &= - t \iint_{\mathbb{R}^6} \nabla w \left(x + tv - \frac{t^{1-\alpha}-1}{1-\alpha} \mathbf{A}_t(v) - \mathcal{X}(t, 0, \tilde{x}, \tilde{v}) \right) f_0(\tilde{x}, \tilde{v}) \, d\tilde{x} d\tilde{v} \\
&= - \frac{1}{t^\alpha} \iint_{\mathbb{R}^6} \nabla w \left(v - \frac{\mathcal{X}(t, 0, \tilde{x}, \tilde{v})-x}{t} - \frac{\mathbf{A}_t(v)(1-\frac{1}{t^{1-\alpha}})}{(1 - \alpha)t^\alpha}\right) f_0(\tilde{x}, \tilde{v})\,  d\tilde{x} d\tilde{v},
\end{aligned}
\]
where in the last step, we have used the homogeneity of \( \nabla w(x) = -\alpha \frac{x}{|x|^{2 + \alpha}} \).
Hence, it follows that 
\[
\begin{aligned}
\left| t \mathbf{E}_f(t, \cdots) - \frac{1}{t^\alpha} \mathbf{A}_t(v) \right| 
&\leq \frac{1}{t^\alpha} \iint_{\mathbb{R}^6} \big|\nabla w(v - y^1) - \nabla w(v - y^0) \big| f_0(\tilde{x}, \tilde{v})\,  d\tilde{x} d\tilde{v} \\
&= \frac{1}{t^\alpha} \int_0^1 \iint_{\mathbb{R}^6} \big|\nabla^2 w(v - y^\theta)\big| \bigg|\frac{dy^\theta}{d\theta}\bigg| f_0(\tilde{x}, \tilde{v})\,  d\tilde{x} d\tilde{v}\, d\theta
\end{aligned}
\]
where
$$ y^\theta = \theta\bigg( \frac{\mathcal{X}(t, 0, \tilde{x}, \tilde{v}) - x}{t} + \frac{\mathbf{A}_t(v)(1-\frac{1}{t^{1-\alpha}})}{(1 - \alpha)t^\alpha}\bigg) + (1-\theta) \mathcal{V}(t, 0, \tilde{x}, \tilde{v}). $$
Note that due to \eqref{eq: forward-in-time characteristic flow bounds-1} and \eqref{eq: derivative bound for A(v)}, we have
$$ \bigg|\frac{dy^\theta}{d\theta}\bigg| \leq \bigg|\frac{\mathcal{X}(t, 0, \tilde{x}, \tilde{v}) - t\mathcal{V}(t, 0, \tilde{x}, \tilde{v}) - x}{t}\bigg| + \frac{|\mathbf{A}_t(v)|}{(1 - \alpha)t^\alpha} \lesssim \frac{|x - \tilde{x}|}{t} + \frac{1}{t^\alpha}, $$
while by \eqref{eq: forward-in-time characteristic flow bounds-2},
$$\bigg|\frac{\partial y^\theta}{\partial \tilde{v}} - \mathbb{I}_3\bigg| \leq \theta \bigg|\frac{\nabla_{\tilde{v}}\mathcal{X}(t, 0, \tilde{x}, \tilde{v})}{t} - \mathbb{I}_3\bigg| + (1-\theta) \big|\nabla_{\tilde{v}}\mathcal{V}(t, 0, \tilde{x}, \tilde{v}) - \mathbb{I}_3\big| \lesssim \eta_* \ll 1.$$
Thus, as in the proof of \eqref{eq: derivative bound for A(v)}, it follows that
\[
\begin{aligned}
\left| t \mathbf{E}_f(t, \cdots) - \frac{1}{t^\alpha} \mathbf{A}_t(v) \right| 
&\lesssim \frac{1}{t^\alpha} \int_0^1 \iint_{\mathbb{R}^6} \big|\nabla^2 w(v - y^\theta)\big| \bigg(\frac{|x - \tilde{x}|}{t} + \frac{1}{t^\alpha}\bigg) f_0(\tilde{x}, \tilde{v})\,  d\tilde{x} d\tilde{v}\, d\theta \\
&\lesssim \frac{\langle x\rangle}{t^{2\alpha}} \int_0^1 \iint_{\mathbb{R}^6} \big|\nabla^2 w(v - \tilde{v})\big|  \Big(\langle x\rangle |f_0|\Big)(\tilde{x}, \tilde{v})\,  d\tilde{x} d\tilde{v}\, d\theta \\
&\lesssim \frac{\langle x \rangle}{t^{2\alpha}} \|\langle \tilde{x} \rangle f_0\|_{L_{\tilde{x}}^1 L_{\tilde{v}}^\infty \cap L_{\tilde{x},\tilde{v}}^1} \lesssim \frac{\eta_*}{t^{2\alpha}} \langle x \rangle.
\end{aligned}
\]
Therefore, by the dispersion bounds \eqref{eq: dispersion estimates} and Lemma \ref{lemma: derivative bound for A(v)}, we conclude
\[
\begin{aligned}
\left\| \frac{\mathbf{F}_1(t)}{\langle x \rangle} \right\|_{L_{x,v}^\infty} 
&\leq \left\| \frac{1}{\langle x \rangle} \left( t \mathbf{E}_f(t, \cdots) - \frac{1}{t^\alpha} \mathbf{A}_t(v) \right) \right\|_{L_{x,v}^\infty} \\
&\quad + \frac{t^{1-\alpha}-1}{1-\alpha} \left\|\partial_t \mathbf{A}_t(v) + \mathbf{E}_f\big( t, \tilde{\Phi}_1^{\textup{ref}}(t)(x,v)\big) \cdot \nabla \mathbf{A}_t(v) \right\|_{L_{x,v}^\infty} \lesssim \frac{\eta_*}{t^{2\alpha}}.
\end{aligned}
\]
\end{proof}

\begin{remark}[Vector field representation of the modified distribution]\label{remark: Vector field representation of the modified distribution}
For each \((x,v) \in \mathbb{R}^3 \times \mathbb{R}^3\), let \(\mathbf{X}(t,x,v)\) be the solution of the initial value problem:
\begin{equation}\label{eq: modified ODE}
\left\{
\begin{aligned}
\dot{\mathbf{X}}(t,x,v) &= \mathbf{F}\big(t, \mathbf{X}(t,x,v)\big), \\
\mathbf{X}(0,x,v) &= (x,v),
\end{aligned}
\right.
\end{equation}
where we denote \(\mathbf{X}(t)(x,v) = \mathbf{X}(t,x,v)\) so that \(\mathbf{X}(t): \mathbb{R}^3 \times \mathbb{R}^3 \to \mathbb{R}^3 \times \mathbb{R}^3\) for each \(t \geq 0\). Then, \(g(t,x,v)\) can be written as the pushforward of \(f_0\) by the vector field, specifically,
$$g(t,x,v) = \big(\mathbf{X}(t)_\# f_0\big)(x,v) = f_0\big(\mathbf{X}(t)^{-1}(x,v)\big),$$
because for \(\varphi \in \mathcal{S}(\mathbb{R}^3 \times \mathbb{R}^3)\),
$$\begin{aligned}
\frac{d}{dt} \big\langle \varphi, \mathbf{X}(t)_\# f_0 \big\rangle &= \frac{d}{dt} \big\langle \varphi\left(\mathbf{X}(t,x,v)\right), f_0(x,v) \big\rangle \\
&= \big\langle \nabla_{(x,v)} \varphi \left(\mathbf{X}(t,x,v)\right) \cdot \mathbf{F}(t, \mathbf{X}(t,x,v)), f_0(x,v) \big\rangle \\
&= \big\langle \nabla_{(x,v)} \varphi \cdot \mathbf{F}(t), \mathbf{X}(t)_\# f_0 \big\rangle = - \big\langle \varphi, \nabla_{(x,v)} \cdot \left( \mathbf{F}(t) \left( \mathbf{X}(t)_\# f_0 \right) \right) \big\rangle \\
&= - \big\langle \varphi, \mathbf{F}(t) \cdot \nabla_{(x,v)} \left( \mathbf{X}(t)_\# f_0 \right) \big\rangle,
\end{aligned}$$
where \(\langle\cdot,\cdot\rangle = \langle\cdot,\cdot\rangle_{L_{x,v}^2}\) and we used that $\mathbf{F}(t)$ is divergence-free in the last step (see \eqref{eq: g eq. divergence-free vector field}).
\end{remark}

With this result in place, we may now prove the proposition.

\begin{proof}[Proof of Proposition \ref{prop: global bounds for modified distribution}]
Due to equation \eqref{eq: modified VR eq-energy estimates} and \eqref{eq: g eq. divergence-free vector field}, we deduce 
$$\frac{d}{dt}\big(\langle x \rangle g(t)\big)
= -\langle x \rangle \mathbf{F}(t) \cdot \nabla_{(x,v)} g(t) = -\mathbf{F}(t) \cdot \nabla_{(x,v)} \big( \langle x \rangle g(t) \big) +  \frac{x}{\langle x \rangle}  \cdot \mathbf{F}_1(t) g(t).$$
Then, by the notation \(\mathbf{X}(t)\) from Remark \ref{remark: Vector field representation of the modified distribution}, this can be expressed\footnote{We evolve from time $t=1$ to avoid the singularity of $\mathbf{F}$ at $t=0$ (see the definition of $\mathbf{F}_1$ in \eqref{eq: modified VR eq-energy estimates-2}).} as
$$
\langle x \rangle g(t) = \mathbf{X}(t)_\# \big( \langle x \rangle g(1) \big) + \int_1^t \mathbf{X}(t - t_1)_\# \bigg( \frac{x}{\langle x \rangle} \cdot \mathbf{F}_1(t_1) g(t_1) \bigg)\, dt_1.
$$
Therefore, using the fact that \(\|\mathbf{X}(t)_\# \phi\|_{L_{x,v}^\infty} = \|\phi\|_{L_{x,v}^\infty}\), we obtain
$$
\|\langle x\rangle g(t)\|_{L_{x,v}^{\infty}} \leq \big\|\langle x \rangle g(1)\big\|_{L_{x,v}^\infty} +  \int_1^t \bigg\| \frac{\mathbf{F}_1(t_1)}{\langle x\rangle }\bigg\|_{L_{x,v}^{\infty}} \|\langle x\rangle g(t_1)\|_{L_{x,v}^{\infty}}\, dt_1.
$$
Subsequently, applying Grönwall's inequality and using the bound \eqref{eq: g eq. vector field bounds-2}, we conclude
\begin{equation}\label{eq: xk weight g L infinity bound}
\|\langle x\rangle g(t)\|_{L_{x,v}^{\infty}} \lesssim \|\langle x\rangle g(1)\|_{L_{x,v}^{\infty}}.
\end{equation}

Next, for the derivatives, we observe that by equation \eqref{eq: modified VR eq-energy estimates} again, \(\langle x \rangle \nabla_{x_j} g(t,x,v)\) satisfies 
$$
\begin{aligned}
\frac{d}{dt} \big( \langle x \rangle \nabla_{x_j} g(t) \big)&= - \langle x \rangle \nabla_{x_j} \Big( \mathbf{F}(t) \cdot \nabla_{(x,v)} g(t) \Big) \\
&= - \langle x \rangle \mathbf{F}(t) \cdot \nabla_{(x,v)} \nabla_{x_j} g(t) - \langle x \rangle \nabla_{x_j} \mathbf{F}(t) \cdot \nabla_{(x,v)} g(t) \\
&= - \mathbf{F}(t) \cdot \nabla_{(x,v)} \big( \langle x \rangle \nabla_{x_j} g(t) \big) + \frac{x}{\langle x \rangle} \cdot \mathbf{F}_1(t) \nabla_{x_j} g(t) \\
&\quad - \langle x \rangle (\nabla_{x_j} \mathbf{F})(t) \cdot \nabla_{(x,v)} g(t),
\end{aligned}
$$
which is equivalent to the integral equation
$$
\begin{aligned}
\langle x \rangle \nabla_{x_j} g(t) &= \mathbf{X}(t)_\# \big( \langle x \rangle \nabla_{x_j} g(1) \big) + \int_1^t \mathbf{X}(t - t_1)_\# \bigg\{ \frac{x}{\langle x \rangle} \cdot \mathbf{F}_1(t_1) (\nabla_{x_j} g)(t_1) \bigg\}\,  dt_1 \\
&\quad - \int_1^t \mathbf{X}(t - t_1)_\# \Big\{ (\nabla_{x_j} \mathbf{F})(t_1) \cdot \big( \langle x \rangle \nabla_{(x,v)} g \big)(t_1) \Big\}\,  dt_1.
\end{aligned}
$$
Hence, using the bounds for \(\mathbf{F}\) (see \eqref{eq: g eq. vector field bounds} and \eqref{eq: g eq. vector field bounds-2}), we obtain
$$
\begin{aligned}
&\|\langle x\rangle \nabla_x g(t) \|_{L_{x,v}^{\infty}}\\
&\lesssim \|\langle x\rangle \nabla_x g(1)\|_{L_{x,v}^{\infty}} + \int_1^t \bigg\|\frac{\mathbf{F}_1(t_1)}{\langle x\rangle }\bigg\|_{L_{x,v}^{\infty}} \|\langle x\rangle  \nabla_x g(t_1) \|_{L_{x,v}^{\infty}}\,  dt_1 \\
&\quad + \int_1^t \|\nabla_x \mathbf{F}_1(t_1)\|_{L_{x,v}^\infty} \|\langle x\rangle \nabla_x g(t_1) \|_{L_{x,v}^{\infty}} + \|\nabla_x \mathbf{F}_2(t_1)\|_{L_{x,v}^\infty} \|\langle x\rangle \nabla_v g(t_1)\|_{L_{x,v}^{\infty}}\,  dt_1 \\
&\lesssim \| \langle x\rangle \nabla_xg(1)\|_{L_{x,v}^{\infty}}+ \int_1^t \frac{\eta_*}{t_1^{2\alpha}} \|\langle x\rangle  \nabla_x g(t_1)\|_{L_{x,v}^{\infty}} + \frac{\eta_*}{t_1^{2+\alpha}} \|\langle x\rangle \nabla_v g(t_1)\|_{L_{x,v}^{\infty}} \,  dt_1.
\end{aligned}
$$
Similarly, due to \eqref{eq: modified VR eq-energy estimates}, we deduce the equation
\begin{equation}\label{eq: x weighted v derivative g}
\begin{aligned}
\frac{d}{dt} \big( \langle x \rangle \nabla_{v_j} g(t) \big)&= - \mathbf{F}(t) \cdot \nabla_{(x,v)} \big( \langle x \rangle \partial_{v_j} g(t) \big) + \frac{x}{\langle x \rangle} \cdot \mathbf{F}_1(t) \nabla_{v_j} g(t) \\
&\quad - \langle x \rangle \nabla_{v_j} \mathbf{F}(t) \cdot \nabla_{(x,v)} g(t),
\end{aligned}
\end{equation}
whose integral form is given by
$$
\begin{aligned}
\langle x \rangle \nabla_{v_j} g(t)
&= \mathbf{X}(t)_\# \big( \langle x \rangle \nabla_{v_j} g(1)\big) + \int_1^t \mathbf{X}(t - t_1)_\# \bigg\{ \frac{x}{\langle x \rangle} \cdot \mathbf{F}_1(t_1) \nabla_{v_j} g(t_1) \bigg\} dt_1 \\
&\quad - \int_1^t \mathbf{X}(t - t_1)_\# \left\{ (\nabla_{v_j} \mathbf{F})(t_1) \cdot \big( \langle x \rangle \nabla_{(x,v)} g \big)(t_1) \right\} dt_1.
\end{aligned}
$$
Subsequently, from the bounds \eqref{eq: g eq. vector field bounds} and \eqref{eq: g eq. vector field bounds-2}, we obtain
$$
\begin{aligned}
&\|\langle x\rangle \nabla_v g(t)\|_{L_{x,v}^{\infty}}\\
&\lesssim \|\langle x\rangle \nabla_v g(1) \|_{L_{x,v}^{\infty}} + \int_1^t \bigg\|\frac{\mathbf{F}_1(t_1)}{\langle x\rangle }\bigg\|_{L_{x,v}^{\infty}} \|\langle x\rangle \nabla_v g(t_1)\|_{L_{x,v}^{\infty}} \, dt_1 \\
&\quad + \int_1^t \|\nabla_v \mathbf{F}_1(t_1)\|_{L_{x,v}^\infty} \|\langle x\rangle \nabla_x g(t_1) \|_{L_{x,v}^{\infty}}+ \|\langle x\rangle \nabla_v \mathbf{F}_2(t_1)\|_{L_{x,v}^{\infty}} \|\langle x\rangle \nabla_v g(t_1) \|_{L_{x,v}^{\infty}} \, dt_1 \\
&\lesssim \|\langle x\rangle \nabla_v g(1)\|_{L_{x,v}^{\infty}} + \int_1^t \frac{\eta_*}{t_1^{\alpha}} \|\langle x\rangle \nabla_x g(t_1)\|_{L_{x,v}^{\infty}} + \frac{\eta_*}{t_1^{2\alpha}} \|\langle x\rangle \nabla_v g(t_1) \|_{L_{x,v}^{\infty}}\,  dt_1.
\end{aligned}
$$
From the above inequality, we observe that if \(\sup_{t \geq 0} \|\langle x\rangle \nabla_x g(t)\|_{L_{x,v}^{\infty}}\) is bounded, then an \(O(t^{1-\alpha})\)-bound is expected for \(\|\langle x\rangle \nabla_v g(t)\|_{L_{x,v}^{\infty}}\) due to the integral of \(\frac{\eta_*}{\langle t_1 \rangle^\alpha} \|\langle x\rangle \nabla_x g(t_1)\|_{L_{x,v}^{\infty}}\). Thus, it is natural to define and estimate the quantity
$$
\mathbf{M}(t) := \|\langle x\rangle\nabla_x g(t)\|_{L_{x,v}^{\infty}} + \frac{1}{t^{1-\alpha}}\|\langle x\rangle\nabla_v g(t)\|_{L_{x,v}^{\infty}}.
$$
Combining the bounds for \(\|\langle x\rangle \nabla_x g(t)\|_{L_{x,v}^{\infty}}\) and \(\|\langle x\rangle \nabla_v g(t)\|_{L_{x,v}^{\infty}}\), we obtain
$$
\mathbf{M}(t) \lesssim \mathbf{M}(1) + \int_1^t \frac{\eta_*}{t_1^{2\alpha}} \mathbf{M}(t_1)\,  dt_1 + \frac{1}{t^{1-\alpha}} \int_1^t \frac{\eta_*}{ t_1^\alpha} \mathbf{M}(t_1) \, dt_1.
$$
Hence, applying Grönwall's inequality yields
\begin{equation}\label{eq: M(t) Gronwall}
\mathbf{M}(t) \lesssim \mathbf{M}(1) \exp \left\{ C \int_1^t \frac{\eta_*}{t_1^{2\alpha}} + \frac{\eta_*}{t^{1-\alpha} t_1^\alpha}\,  dt_1 \right\}
\lesssim \mathbf{M}(1).
\end{equation}

It remains to estimate $\|\langle x\rangle g(1)\|_{L_{x,v}^{\infty}}$ (see \eqref{eq: xk weight g L infinity bound}) and $\mathbf{M}(1)$ and show that it is controlled by the initial data, and thus sufficiently small. Indeed, by the definitions of $g$ (see \eqref{eq: g(t) definition}) and the smallness bound for $\mathbf{A}_t$ (see \eqref{eq: derivative bound for A(v)}), we have
$$\begin{aligned}
\mathbf{M}(1) &= \left\|\langle x\rangle \nabla_{(x,v)}\left\{ f\left(1, x + v - \frac{1}{1-\alpha} \mathbf{A}_1(v), v \right)\right\}\right\|_{L_{x,v}^{\infty}}\\
&\lesssim\left\|\langle x\rangle (\nabla_{(x,v)}f)\left(1, x + v - \frac{1}{1-\alpha} \mathbf{A}_1(v), v \right)\right\|_{L_{x,v}^\infty}\lesssim\big\|\langle x-v\rangle \nabla_{(x,v)}f(1)\big\|_{L_{x,v}^\infty}.
\end{aligned}$$
Subsequently, by the notation for the backward-in-time flow in \eqref{eq: Hamiltonian ODE0} and its derivative bounds in \eqref{eq: derivative characteristic bound}, it follows that
$$\begin{aligned}
\mathbf{M}(1)&\lesssim\Big\|\langle x-v\rangle \nabla_{(x,v)}\Big(f_0\big(\mathcal{X}(0,1,x,v), \mathcal{V}(0,1,x,v)\big)\Big)\Big\|_{L_{x,v}^\infty}\\
&\lesssim \Big\|\langle x-v\rangle \big(\nabla_{(x,v)}f_0\big)\big(\mathcal{X}(0,1,x,v), \mathcal{V}(0,1,x,v)\big)\Big\|_{L_{x,v}^\infty}\\
&=\big\|\langle \mathcal{X}(1,0,x,v)-\mathcal{V}(1,0,x,v)\rangle \nabla_{(x,v)}f_0\big\|_{L_{x,v}^\infty},
\end{aligned}$$
where in the last step, we have changed variables by $$(x,v)\mapsto \Phi(1)(x,v)=(\mathcal{X}(1,0,x,v), \mathcal{V}(1,0,x,v)).$$
Note also that by equation \eqref{eq: characteristic flow, integral form} and the force field bound \eqref{eq: dispersion estimates}, we have 
$$\big|\mathcal{X}(1,0,x,v)-\mathcal{V}(1,0,x,v)-x\big| \leq \int_0^1 t_1\|\mathbf{E}_f(t_1)\|_{L_x^\infty} \, dt_1\lesssim\eta_*.$$
Therefore, we obtain 
$$\mathbf{M}(1)\lesssim \big\|\langle x\rangle \nabla_{(x,v)}f_0\big\|_{L_{x,v}^\infty}\lesssim\eta_*.$$
Similarly, we can show that
$$\|\langle x\rangle g(1)\|_{L_{x,v}^{\infty}}=\big\|\langle \mathcal{X}(1,0,x,v)-\mathcal{V}(1,0,x,v)\rangle f_0\big\|_{L_{x,v}^\infty}\lesssim \|\langle x\rangle f_0\|_{L_{x,v}^\infty}\lesssim\eta_*.$$
Inserting these bounds within \eqref{eq: M(t) Gronwall} and \eqref{eq: xk weight g L infinity bound} completes the proof of the proposition.
\end{proof}

\section{Proof of modified scattering}\label{sec: proof of modified scattering}

To prove the main theorem (Theorem~\ref{theorem: modified scattering}), we first take $t \geq 1$ throughout, define
$$f^+(x,v):= f_0\big((\mathcal{W}^{\mathrm{mod},+})^{-1} (x,v)\big),$$
and
decompose the difference
$$
\begin{aligned}
&\big\|f\big(t, \Phi^{\mathrm{ref}}(t)(x,v)\big)-f^+(x,v)\big\|_{L_{x,v}^\infty}\\
&\leq \big\|f\big(t, \Phi^{\mathrm{ref}}(t)(x,v)\big)-f\big(t, \tilde{\Phi}^{\mathrm{ref}}(t)(x,v)\big)\big\|_{L_{x,v}^\infty} + \|g(t)-f^+\|_{L_{x,v}^\infty}\\
&=\textup{(I)} + \textup{(II)},
\end{aligned}
$$
where $g(t,x,v) = f(t, \tilde{\Phi}^{\mathrm{ref}}(t)(x,v))$ denotes the modified distribution defined in \eqref{eq: g(t) definition}. We estimate $\textup{(I)}$ and $\textup{(II)}$ separately.

For $\textup{(I)}$, we first note that 
$\tilde{\Phi}^{\mathrm{ref}}(t)$ and $\Phi^{\mathrm{ref}}(t)$ share the same momentum component, i.e.
\[
\tilde{\Phi}_2^{\mathrm{ref}}(t)(x,v)=\Phi_2^{\mathrm{ref}}(t)(x,v)=v.
\]
On the other hand, within the position component we have the decaying difference bound
$$
|\Phi_1^{\mathrm{ref}}(t)(x,v)- \tilde{\Phi}_1^{\mathrm{ref}}(t)(x,v)|
=\frac{t^{1-\alpha}-1}{1-\alpha}|\mathbf{A}_t(v)-\mathbf{A}_\infty(v)|
\lesssim \frac{\eta_*}{t^{2\alpha-1}},
 $$
due to the convergence estimate \eqref{eq: convergence estimate for A_t(v)}. Thus, together with the global-in-time bound for $\|\nabla_x f(t)\|_{L_{x,v}^\infty} = \|\nabla_x g(t)\|_{L_{x,v}^\infty}$ provided by Proposition \ref{prop: global bounds for modified distribution}, this yields
\begin{equation}\label{eq: modified scattering proof, part 1}
(\textup{I})\lesssim \frac{\eta_*}{t^{2\alpha-1}} \|\nabla_x f(t)\|_{L_{x,v}^\infty}
\lesssim \frac{\eta_*^2}{t^{2\alpha-1}}.
\end{equation}

For $\textup{(II)}$, we note that
$$
g(t,x,v) = f_0\Big(\big(\Phi(t)^{-1} \circ \tilde{\Phi}^{\mathrm{ref}}(t)\big)(x,v)\Big)
= f_0\big(\mathcal{W}^{\mathrm{mod}}(t)^{-1}(x,v)\big),
$$
and thus, the scattering state can be expressed in terms of $g$ as
$$
f^+(x,v)= f_0\big((\mathcal{W}^{\mathrm{mod},+})^{-1}(x,v)\big)
= g\Big(t, \big(\mathcal{W}^{\mathrm{mod}}(t) \circ (\mathcal{W}^{\mathrm{mod},+})^{-1}\big)(x,v)\Big).
$$
Therefore, using the fact that $\mathcal{W}^{\textup{mod},+}$ is one-to-one, we can write 
$$
\textup{(II)} = \big\| g\big(t,\mathcal{W}^{\textup{mod},+}(x,v)\big)
- g\big(t,\mathcal{W}^{\textup{mod}}(t)(x,v)\big) \big\|_{L_{x,v}^\infty}.
$$
Hence, applying the fundamental theorem of calculus, we obtain
\begin{equation}\label{eq: g f+ difference estimate}
\begin{aligned}
\textup{(II)} &\leq \int_0^1 
\big\| \nabla_x g\big(t,\mathcal{W}^\theta(t)(x,v)\big)
\cdot \big(\mathcal{W}_1^{\textup{mod},+} - \mathcal{W}_1^{\textup{mod}}(t)\big)(x,v) 
\big\|_{L_{x,v}^\infty} \, d\theta \\
&\quad + \int_0^1 
\big\| \nabla_v g\big(t,\mathcal{W}^\theta(t)(x,v)\big)
\cdot \big(\mathcal{W}_2^{\textup{mod},+} - \mathcal{W}_2^{\textup{mod}}(t)\big)(x,v)
\big\|_{L_{x,v}^\infty} \, d\theta \\
&= \textup{(II)}_1 + \textup{(II)}_2,
\end{aligned}
\end{equation}
where
$$
\mathcal{W}^\theta(t) := \theta \mathcal{W}^{\textup{mod},+}
+ (1-\theta) \mathcal{W}^{\textup{mod}}(t).
$$
By the convergence bound for the position component of the wave operator (Proposition \ref{prop: modified wave operator} $(1)$), we first obtain
$$
\textup{(II)}_1 \lesssim \frac{\eta_*}{t^{2\alpha-1}} \int_0^1 \big\|\langle x\rangle \nabla_x g\big(t,\mathcal{W}^\theta(t)(x,v)\big)\big\|_{L_{x,v}^\infty} \, d\theta.
$$
Additionally, note that
$$
\langle x\rangle \sim \big\langle \mathcal{W}_1^\theta(t)(x,v)\big\rangle,
$$
because by Proposition \ref{prop: modified wave operator} (2), we have
$$
|\mathcal{W}_1^\theta(t)(x,v)-x| \leq \theta |\mathcal{W}_1^{\textup{mod},+}(x,v)-x| + (1-\theta)|\mathcal{W}_1^{\textup{mod}}(t)(x,v)-x| \lesssim \eta_* \langle x \rangle.
$$
Thus, it follows that
\begin{equation}\label{eq: g f+ difference estimate, part1}
\begin{aligned}
\textup{(II)}_1 &\lesssim \frac{\eta_*}{t^{2\alpha-1}} \int_0^1 \big\|\big\langle \mathcal{W}_1^\theta(t)(x,v)\big\rangle \nabla_x g\big(t,\mathcal{W}^\theta(t)(x,v)\big)\big\|_{L_{x,v}^\infty} \, d\theta \\
&\lesssim \frac{\eta_*}{t^{2\alpha-1}} \|\langle x\rangle \nabla_x g(t)\|_{L_{x,v}^\infty} \lesssim \frac{\eta_*^2}{t^{2\alpha-1}},
\end{aligned}
\end{equation}
due to the uniform bound for $\|\langle x\rangle \nabla_x g(t)\|_{L_{x,v}^\infty}$ (Proposition~\ref{prop: global bounds for modified distribution}).
On the other hand, by the convergence bound for the momentum part of the wave operator (Proposition \ref{prop: modified wave operator} $(1)$) and the $O(t^{1-\alpha})$ bound for $\|\nabla_v g(t)\|_{L_{x,v}^\infty}$ (Proposition~\ref{prop: global bounds for modified distribution}), we obtain
\begin{equation}\label{eq: g f+ difference estimate, part2}
\textup{(II)}_2 \lesssim  \frac{\eta_*}{t^\alpha} \int_0^1 \big\|\nabla_v g\big(t,\mathcal{W}^\theta(t)(x,v)\big)\big\|_{L_{x,v}^\infty} \, d\theta \lesssim \frac{\eta_*}{t^\alpha} \|\nabla_v g(t)\|_{L_{x,v}^\infty} \lesssim \frac{\eta_*^2}{t^{2\alpha-1}}.
\end{equation}
Therefore, combining \eqref{eq: modified scattering proof, part 1}-\eqref{eq: g f+ difference estimate, part2}, the proof of the main result \eqref{eq: modified scattering} is complete.

\begin{remark}[Technical comment]\label{remark: modified scattering proof key remark}
In \eqref{eq: g f+ difference estimate}, it is crucial to estimate
$$
\nabla_{(x,v)} g\big(t,\mathcal{W}^\theta(t)(x,v)\big)
\cdot
\big(\mathcal{W}^{\mathrm{mod},+} - \mathcal{W}^{\mathrm{mod}}(t)\big)(x,v)
$$
by separating the position and momentum components:
$$
\begin{aligned}
&\nabla_x g\big(t,\mathcal{W}^\theta(t)(x,v)\big)
\cdot
\big(\mathcal{W}_1^{\mathrm{mod},+} - \mathcal{W}_1^{\mathrm{mod}}(t)\big)(x,v),\\
&\nabla_v g\big(t,\mathcal{W}^\theta(t)(x,v)\big)
\cdot
\big(\mathcal{W}_2^{\mathrm{mod},+} - \mathcal{W}_2^{\mathrm{mod}}(t)\big)(x,v).
\end{aligned}
$$
This separation is essential because we only have the growing bound 
$|\nabla_v g(t)| = O(t^{1-\alpha})$ from Proposition~\ref{prop: global bounds for modified distribution}. 
This growth can be compensated for by the faster convergence 
\(\mathcal{W}_2^{\mathrm{mod}}(t) \to \mathcal{W}_2^{\mathrm{mod},+}\) compared to 
\(\mathcal{W}_1^{\mathrm{mod}}(t) \to \mathcal{W}_1^{\mathrm{mod},+}\), as seen from  Proposition \ref{prop: modified wave operator} $(1)$, 
\eqref{eq: g f+ difference estimate, part1}, and 
\eqref{eq: g f+ difference estimate, part2}.
Without separating the components and exploiting the faster convergence of the momentum part of the modified wave operator, estimating
$$
\big|\nabla_{(x,v)} g\big(t,\mathcal{W}^\theta(t)(x,v)\big)\big|
\,\big|\big(\mathcal{W}^{\mathrm{mod},+} - \mathcal{W}^{\mathrm{mod}}(t)\big)(x,v)\big|
$$
would not allow us to include the range $\frac{1}{2} < \alpha < 1$.
\end{remark}

\end{document}